\documentclass[12pt]{amsart}

\def \thus{:\;\;}

\def\eps{{\varepsilon}}

\def\Card{{\rm Card}}

\def\Prob{{\mathbb{P}}}

\def\Vol{{\rm Vol}}
\def\EXP{{\mathbb{E}}}

\def\BAN{\mathbb{B}}

\def\Tor{\mathbb{T}}

\def\reals{\mathbb{R}}

\def\integers{\mathbb{Z}}

\def\ba{\mathbf{a}}

\def\bn{\mathbf{n}}

\def\brH{{\bar H}}

\def\brL{{\bar L}}

\def\brQ{{\bar Q}}

\def\brY{{\bar Y}}

\def\brn{{\bar n}}

\def\brt{{\bar t}}

\def\brx{{\bar x}}

\def\brrho{{\bar\rho}}

\def\cC{\mathcal{C}}

\def\cI{\mathcal{I}}

\def\cJ{\mathcal{J}}

\def\cF{\mathcal{F}}

\def\cG{\mathcal{G}}

\def\cE{\mathcal{E}}

\def\cK{\mathcal{K}}

\def\cP{\mathcal{P}}

\def\cQ{\mathcal{Q}}

\def\cT{\mathcal{T}}

\def\cY{\mathcal{Y}}

\def\fP{\mathfrak{P}}

\def\fg{\mathfrak{g}}

\def\fp{\mathfrak{p}}

\def\hB{{\hat B}}

\def\hD{{\hat D}}

\def\hS{{\hat S}}

\def\hg{{\hat g}}

\def\hdelta{{\hat\delta}}

\def\tC{{\tilde C}}

\def\tH{{\tilde H}}

\def\tS{{\tilde S}}

\def\ty{{\tilde y}}

\def\tnu{{\tilde\nu}}

\makeatother

\usepackage{xcolor}
\usepackage{enumerate}
\usepackage{ulem}
\usepackage[margin=3cm]{geometry}
\usepackage{amsthm, bbm}
\usepackage{mathrsfs}
\usepackage{graphicx, psfrag, epstopdf}
\usepackage{hyperref}
\usepackage{amssymb}

\theoremstyle{definition}
\newtheorem{definition}{Definition}[section]
\newtheorem{theorem}[definition]{Theorem}
\newtheorem{lemma}[definition]{Lemma}
\newtheorem{proposition}[definition]{Proposition}
\newtheorem{corollary}[definition]{Corollary}
\newtheorem{remark}[definition]{Remark}
\newtheorem{example}[definition]{Example}

\numberwithin{equation}{section}
\definecolor{OliveGreen}{rgb}{0,0.6,0}

\def\one{{\mathbbm{1}}}
\def\DS{\displaystyle}
\newcommand{\N}{{\mathbb{N}}}
\author{D. Dolgopyat, C. Dong, A. Kanigowski, and P. N\'andori}
\title{Mixing properties of generalized $T, T^{-1}$ transformations}
\begin{document}
\maketitle
\begin{abstract}
We study mixing properties of generalized $T, T^{-1}$ transformations. We discuss two mixing mechanisms.
In the case the fiber dynamics is mixing, it is sufficient that the driving cocycle is small with small probability.
In the case the fiber dynamics is only assumed to be ergodic, one needs to use the shearing properties
of the cocycle. Applications include the Central Limit Theorem for sufficiently fast mixing systems
and the estimates on deviations of ergodic averages.
\end{abstract}
\tableofcontents

\section{Introduction}
An important discovery made in the last century is that deterministic systems can exhibit chaotic behavior.
Currently there are many examples of systems exhibiting a full array of chaotic properties including 
Bernoulli property, exponential decay of correlations and central limit theorem (see e.g. 
\cite{BDV05, Bow75, CM06, You98}). 
Systems which satisfy only some of the above properties are less
understood. In fact, it is desirable to have more examples of such systems in order to understand 
the full range of possible behaviors of partially chaotic systems.

Generalized $T, T^{-1}$ transformations are a rich source of examples in probability and ergodic theory.
In fact, they were used to exhibit examples of systems with unusual limit laws \cite{KS79, CC17},
central limit theorem with non standard normalization \cite{Bolt89}, K but 
non Bernoulli systems
in abstract \cite{Kal82} and smooth setting in various dimensions \cite{Kat80, Rud88, KRHV},
very weak Bernoulli but not weak Bernoulli partitions \cite{dHKSS},
slowly mixing systems \cite{dHS97, LB06}, systems with multiple Gibbs measures \cite{EL04, MN79}.

A comprehensive survey of probabilistic version of $T, T^{-1}$ transformations, which is a random walk in random scenery,
is contained in \cite{dHS06}. On the other hand, there are no works addressing how statistical properties of
$T, T^{-1}$ transformations depend on the properties of the base and the fiber dynamics. Our paper provides a first step in this direction by investigating
mixing properties of $T, T^{-1}$ transformations.

Let us explain what we mean by smooth $T, T^{-1}$ transformations.
Let $X, Y$ be compact manifolds, $f:X\to X$ be a smooth map preserving 
a measure $\mu$ and $G_t: Y\to Y$ be a $d$ parameter
flow on $Y$ preserving a measure $\nu.$ Let $\tau: X\to\reals^d$ be a smooth map.
We study the following map $F: (X\times Y)\to (X\times Y)$
$$ F(x,y)=(f(x), G_{\tau(x)}y). $$
Note that $F$ preserves the measure $\zeta=\mu\times \nu$ and that 
$$ F^N(x,y)=(f^N x, G_{\tau_N(x)} y)\quad\text{where}\quad 
\tau_N(x)=\sum_{n=0}^{N-1} \tau(f^n x). $$

Clearly both mixing of $f$ and ergodicity of $G$ are necessary for $F$ to be mixing. Under these assumptions there are two mechanisms
for $F$ to be mixing.

(1) If $G$ itself is mixing then it is enough to ensure that $\tau_N$ does not take small values with large probability
(cf. \cite{dHS97, LB06}).

(2) On the other hand if we only assume that $G$ is ergodic then we need to rely on shearing properties of $\tau$ to ensure
that $\tau_N$ is uniformly distributed in boxes of size 1. This can be done by assuming various extension of the Central Limit Theorem
(cf. \cite{Br05, DLN19}).

Abstract results detailing sufficient conditions for each of the two mechanisms
described above  are presented in Section \ref{ScMixAbs}.
Estimates on the rates of mixing of $F$ under the assumption that $G$ is mixing are given in Section \ref{ScMix-MF}. 
In Section \ref{ScCLT}, we prove the Central Limit Theorem
in case $F$ mixes sufficiently quickly.
Section \ref{ScMix-EF}
contains mixing estimates in case
$G$ is only assumed to be ergodic (however, we need much stronger assumptions on the base map $f$). 
The results presented in Sections
\ref{ScMix-MF}--\ref{ScMix-EF} rely on preliminary facts contained in
Section \ref{ScBack}. In Section \ref{ScToral}, we discuss 
several examples which
require a combination of ideas from Sections \ref{ScMix-MF}
and \ref{ScMix-EF}. 
Section \ref{ScDeviations} presents application of our mixing results to deviations of ergodic averages
and also contains a survey of
examples of systems satisfying various assumptions required in our results. 
We will have some strong assumptions that
are sometimes non-trivial to check. In the appendix,
we check one of our assumptions for an important example, namely the 
anticoncentration large deviation bounds for subshifts
of finite type. This result
may be interesting outside of the scope of the present work.

We also mention that in a followup  paper \cite{DDKN-Flex}
we provide a description of
further statistical properties of the generalized $T, T^{-1}$ transformation, using the mixing
bounds obtained in the present paper.
\\

{\bf Acknowledgements:} D.\ D.\ was partially supported by the NSF grant
DMS-1956049,
A.\ K.\ was partially supported by the NSF grant DMS-1956310,
P.\ N.\ was partially supported by the NSF grants DMS-1800811 and DMS-1952876.


\section{Local Limit Theorem and Mixing}
\label{ScMixAbs}
For 
a function $A\in L^1(X,\mu)$ we denote $\mu(A(\cdot)):=\int_{X}A(x)\,d\mu$.

\begin{definition}
\label{DfnMixLLT}  
$\tau$ satisfies {\em mixing LLT} if there exist sequences 
$(L_n)_{n\in \N}\subset \reals,$ $(D_n)_{n\in \N}\subset \reals^d$ 
and a bounded probability density $\fp$  on $\reals^d$ such that
for any sequence $(\delta_n)_{n \in \mathbb N} \subset \reals$, with
$\displaystyle \lim_{n\to\infty} \delta_n = 0$,
$(z_n)_{n\in \N}\subset \reals^d$ such that 
$ |\frac{z_n}{L_n}- z| < \delta_n$ for any cube $\cC\subset \reals^d$ and any continuous functions $A_0, A_1:X\to\reals$,
$$ \lim_{n\to\infty} L_n^d \mu\Big(A_0(\cdot) A_1(f^n\cdot) \one_\cC(\tau_n-D_n-z_n)\Big)=
\fp(z) \mu(A_0) \mu(A_1) \Vol(\cC),$$
and the convergence is uniform  once 
$(\delta_n)_{n \in \mathbb N}$ is fixed and 
$A_0, A_1, z$ range over compact subsets of $C(X), C(X)$ and $\reals^d$ respectively.
\end{definition}

\begin{definition}
We say that, $\tau$ satisfies {\em mixing multiple LLT} if for each $m\in \N$, any sequence $(\delta_n)_{n \in \mathbb N} \subset \reals$ with
$\displaystyle \lim_{n\to\infty} \delta_n = 0$, and any
family of sequences $(z_n^{(1)}, \dots, z_n^{(m)})_{n\in \N}$ with
$ |\frac{z_n^{(j)}}{L_n} - z^{(j)}| < \delta_n$,  any cubes $\{\cC_j\}_{j\leq m}\subset \reals^d$ and continuous functions $A_0, \dots ,A_m:X\to\reals$, for any sequences $n_k^{(1)}, \dots, n_k^{(m)}\in \N$
 such that $n^{(j)}_{k}-n^{(j-1)}_{k}\geq \delta_k^{-1}
$ (with $n^{(0)}_k=0$),
$$ \lim_{k\to\infty} \left(\prod_{j=1}^m L^d_{n^{(j)}_{k}-n^{(j-1)}_{k}}\right) 
\mu\left(\prod_{j=0}^m A_j
\left(f^{n^{(j)}_{k}} \cdot \right) 
\prod_{j=1}^m \one_{\cC_j}\left(\tau_{n^{(j)}_{k}}-D_{n^{(j)}_{k}}-z_{n^{(j)}_{k}}^{(j)}\right)\right)$$
$$=
\prod_{j=0}^m \mu(A_j) \prod_{j=1}^m \fp\left(z^{(j)}-z^{(j-1)}\right) \prod_{j=1}^m \Vol(\cC_j)
$$
where $z^{(0)}=0.$ Moreover, the convergence is uniform  once 
$(\delta_n)_{n \in \mathbb N}$ is
fixed, 
 $A_0,\dots, A_m$ range over compact subsets of $C(X)$ and $z^{(j)}$ range over 
a compact subset of $\reals^d$ for every $j \leq  m$.
\end{definition}

\begin{remark}
\label{remark:drift}
We note that $\tau$ is bounded and consequently $\tau_n/n$
is bounded, too. Thus if the mixing 
LLT holds, then $L_n < Cn$. We assume that $D_n = n \mu(\tau)$. In case $\mu(\tau) = 0$, we say
that $\tau$ has zero drift. 
\end{remark}

\begin{remark}
By Portmanteau theorem on vague convergence, the mixing LLT is equivalent to saying that
for all continuous functions $A_0, A_1: X\to \reals$ for any 
compactly supported almost everywhere continuous
function $\phi:\reals^d\to \reals$  for any sequence $z_N$ such that 
$| \frac{z_N}{L_N}- z| < \delta_n$, we have
\begin{equation}
\label{EqPortmanteau}
 \lim_{n\to\infty} L_n^d \mu\Big(A_0(\cdot) A_1(f^n\cdot) \phi(\tau_n-D_n-z_n)\Big)=
\fp(z) \mu(A_0) \mu(A_1) \int_{\reals^d} \phi(t) dt 
\end{equation}
and the convergence is uniform if
$A_0, A_1$ range over compact subsets of $C(X)$ and 
 $z$ ranges over a compact subset of $\reals^d.$
 A similar remark applies to the multiple mixing LLT.
\end{remark}

\begin{theorem}
\label{ThLLTMix}
Suppose that $(G_t)$ is ergodic.

(a) If $\tau$ satisfies the mixing LLT then $F$ is mixing.

(b) If $\tau$ satisfies the mixing multiple LLT then $F$ is multiple mixing.
\end{theorem}

\begin{proof}
(a)
For $i=1,2$, let $\Phi_i(x,y)=A_i(x)B_i(y)$ be a continuous function on $X\times Y$. Since linear combinations of products as above are dense in $L^2(\mu\times\nu),$
it suffices to show that for every $\epsilon>0$ there exists $N_0\in \mathbb N$ 
such that for every $N\geq N_0$, we have
\begin{equation}\label{e-0}
\Big|\int_{X\times Y} \Phi_1(x,y)\Phi_2(F^N(x,y))d(\mu\times\nu)-\mu(A_1)\mu(A_2)\nu(B_1)\nu(B_2)\Big|<\epsilon.
\end{equation}

Let $\rho(t):=\int_Y B_1(y)B_2(G_{t}y)d\nu(y)$. Note that \begin{equation}\label{e-1}\int_{X\times Y} \Phi_1(x,y)\Phi_2(F^N(x,y))d(\mu\times\nu)=\int_X A_1(x)A_2(f^N(x))\cdot\rho(\tau_N(x))d\mu(x).\end{equation}

Let $\delta=\delta(\epsilon)>0$ be small with respect to $\epsilon$, and $I_0\subset \mathbb R^d$ be a cube of volume $\delta^d$, centered at 0. Consider a (disjoint) cover of $\mathbb R^d $ by a union of small cubes 
$\{I_j\}$, where $I_j$ is a translation of $I_0$ and let  $t_j$ denote the center of $I_j$.
Now let ${\bf B_\ell}\subset \mathbb R^d$ be a ball centered at 0 with radius $\ell$, and denote $S_\ell:=\{j:I_j \cap\bf B_\ell\neq\emptyset\}$. 
By the mixing LLT (with $A_0=A_1=1$) it follows that  there exists $K=K(\epsilon)$ and $N_0'\in \N$ such that for every $N\geq N_0'$, 
$$
\mu\Big(\{x\in X:|\tau_N-D_N|>KL_N/2\}\Big)<\epsilon/2.
$$
Let $\hS_1:=S_{KL_N}$. Therefore (see \eqref{e-0} and \eqref{e-1}) it is enough to show that
\begin{equation}\label{e-1.5}
\Big|\sum_{j\in \hS_{1}}\int A_1(x)A_2(f^N(x))\cdot\rho(\tau_N(x))
\one_{I_j + D_N}(\tau_N(x))
d\mu(x)-
\mu(A_1)\mu(A_2)\nu(B_1)\nu(B_2)\Big|<\epsilon/2.
\end{equation}

If $\delta$ is small enough (using continuity of $(G_t)$), the above sum is, up to an error less than 
$\epsilon/16$, equal to
\begin{equation}\label{e-2}
\sum_{j\in \hS_{1}}\rho(D_N+t_j)\mu\Big(A_1(\cdot)A_2(f^N(\cdot))\one_{I_j}(\tau_N(\cdot)- D_N)\Big).
\end{equation}
 By the definition of mixing LLT  (with $A_1, A_2$, $\cC=I_0$ and $z=t_j$), and since the number of $j's$ such that $j\in \hat S_1$ is bounded above by $C(\delta,\epsilon)L_N^d$ there exists $N_1=N_1(\epsilon,\delta)\in \N$ such that for every $N\geq N_1$, the above expression is, up to an error less than  $\epsilon/16$, equal to  
 \begin{equation}\label{e-3}\sum_{j\in \hat S_{1}}\frac{1}{L_N^d}\Vol(I_0)
 \fp\left(\frac{t_j}{L_N}\right)\mu(A_1)\mu(A_2)\rho(D_N+t_j).\end{equation}
 
Enlarging  $K$ and $N$, if necessary, we can guarantee that 
\begin{equation}\label{to1}
\left|\sum_{j\in \hat S_{1}}\frac{1}{L_N^d}\Vol(I_j)\fp\left(\frac{t_j}{L_N}\right)-1\right|
<\frac{\epsilon}{16}.
\end{equation}

Now, fix $R>0$ and for $c\in {\bf B}_R$, let 
$$\alpha(c):=\sum_{j\in \hat S_{1}}\frac{1}{L_N^d}\Vol(I_0)\fp\left(\frac{t_j}{L_N}\right)\rho(D_N+t_j+c).$$ 
We claim that there exists $N_2=N_2(R)$ such that for $N\geq N_2$, we have 
$$
|\alpha(c)-\alpha(0)|<\epsilon/16.
$$
Indeed, let $k$ be such that $c\in I_k$, then $|t_k|\le 
R+1$ and $|t_k-c|\le \delta$, by choosing $\delta\ll\epsilon$ small enough, and $N_2$ large so that $\frac{R+1}{L_{N_2}}\le \delta$, we have
\begin{align}
&|\alpha(c)-\alpha(0)|\le |\alpha(c)-\alpha(t_k)|+|\alpha(t_k)-\alpha(0)|\nonumber \\&\le \frac{\Vol(I_0)}{L_N^d}\sum_{j\in \hat S_{1}}\fp(\frac{t_j}{L_N})|\rho(D_N+t_j+c)-\rho(D_N+t_j+t_k)|+\nonumber\\&\frac{\Vol(I_0)}{L_N^d}\sum_{j\in \hat S_{1}}\left|\fp(\frac{t_j}{L_N})-\fp(\frac{t_j-t_k}{L_N})\right||\rho(D_N+t_j)|+\nonumber\\&\frac{\Vol(I_0)}{L_N^d}\sum_{j\in \hat S_1:|t_j-t_k|\ge KL_N}\fp(\frac{t_j}{L_N})|\rho(D_N+t_j+t_k)|\nonumber\\& (\star)\le \; C_1(\fp,\rho)|t_k-c|+C_2(\fp,\rho,K)R/L_N+K^dC(\rho)R/L_N\nonumber\\&\le\epsilon/64+\epsilon/64+\epsilon/64<\epsilon/16,\nonumber
\end{align}
where for the inequality $(\star)$, the first term is due to the fact that $\rho$ is 
continuous on $t$ and \eqref{to1}, the second term is due to continuity of $\fp$ and the choice of $N_2$ (that is,
$\frac{R+1}{L_{N}}\le\delta$), and the last term contains a sum of $K^dRL_N^{d-1}$ many terms and hence $\le K^dC(\rho)R/L_N$.

Therefore 
\begin{equation}\label{e-3.5}\Big|\alpha(0)-\frac{1}{\Vol({\bf B}_R)}\int_{c\in {\bf B}_R}\alpha(c)dc\Big|<\epsilon/16.
\end{equation}
Now by the ergodicity of $G$ and the mean 
ergodic theorem for the $G$-action, there exist a subset $Y_0\subset Y$ with $\nu(Y_0)\ge 1-\frac{\epsilon}{32C_3^2}$ and $R_0>0$, such that for any $y\in Y_0$ and $R\ge R_0$, $$\left|\frac{1}{\Vol({\bf B}_R)}\int_{t\in {\bf B}_R}B_2(G_ty)dt-\nu(B_2)\right|<\frac{\epsilon}{32C_3}.$$Here the constant $C_3:=10\max_{y\in Y}\{|B_1(y)|,|B_2(y)|\}$. Hence for any $t$, if $R\ge R_0$,
\begin{align}\label{e-4}
&\left|\frac{1}{\Vol({\bf B}_R)}\int_{c\in {\bf B}_R}\rho(t+c)dc- \nu(B_1)\nu(B_2)\right|\\\le& \left|\int_{G_{-t}(Y_0)}B_1(y)\left(\frac{1}{\Vol({\bf B}_R)}\int_{c\in {\bf B}_R}B_2(G_{t+c}y)dc-\nu(B_2)\right)d\nu(y)\right|\nonumber \\&+\int_{Y\backslash G_{-t}(Y_0)}|B_1(y)|\left|\frac{1}{\Vol({\bf B}_R)}\int_{c\in {\bf B}_R}B_2(G_{t+c}y)dc-\nu(B_2)\right|d\nu(y)\nonumber \\\le& \max\{|B_1|\}\frac{\epsilon}{32C_3}+\max\{|B_1|\}\max\{|B_2|\}2(1-\nu(Y_0))\le\frac{\epsilon}{16}.\nonumber 
\end{align}

Note that \eqref{e-3} is equal to $\mu(A_1)\mu(A_2)\alpha(0)$. By \eqref{e-3.5} and \eqref{e-4}, $\mu(A_1)\mu(A_2)\alpha(0)$ is,  up to an error less than $\epsilon/8$, equal to 
$$
\mu(A_1)\mu(A_2)\nu(B_1)\nu(B_2)
\left[\sum_{j\in \hat S_{1}}\frac{1}{L_N^d}\Vol(I_j)\fp\left(\frac{t_j}{L_N}\right)\right].
$$

Combining the estimates \eqref{to1}, (\ref{e-2}) and (\ref{e-3}) we obtain (\ref{e-1.5}) (and consequently \eqref{e-0}), completing the proof.

(b) The proof is essentially the same as that for (a), therefore we leave it to the reader. 
\end{proof}

\section{Background}
\label{ScBack}

\begin{definition}
We say that $G$ is {\em mixing with rate $\psi(t)$ on a space $\BAN$} if 
\begin{equation}
\label{Eq2Mix}
 \left| \int B_1(y) B_2(G_t y) d\nu(y)-\nu(B_1)\nu(B_2) \right|\leq 
C \psi(t) \|B_1\|_\BAN \|B_2\|_\BAN. 
\end{equation} 
We call $G$ {\em exponentially mixing} if \eqref{Eq2Mix} holds with $\BAN=C^r$ for some $r>0$ and 
$\psi(t)=e^{-\delta\|t\|}$ for some $\delta>0.$

We call $G$ {\em polynomially mixing} if \eqref{Eq2Mix} holds with $\BAN=C^r$ for some $r>0$ and 
$\psi(t)=\|t\|^{-\delta}$ for some $\delta>0.$

We call $G$ {\em rapidly mixing} if for each $m$ there exists $r$ such that 
 \eqref{Eq2Mix} holds with $\BAN=C^r$ and $\psi(t)=\|t\|^{-m}.$ 

These definitions extend to maps (such as to $f$ and $F$) in the 
natural way.
\end{definition}

\begin{definition}
$\tau$ satisfies {\em exponential large deviation bounds}, 
if for each $\eps>0$ there exist $C$ and $\delta>0$ 
such that for any $N\in\N$,
\begin{equation}
\label{EqLD-Exp}
 \mu\left(
\left\|
\frac{\tau_N}{N}- \mu(\tau)
\right\|\geq \eps\right)\leq C e^{-\delta N}. 
\end{equation} 
\medskip
$\tau$ satisfies {\em polynomial large deviation bounds},
if for each $\eps>0$ there exist $C$ and $\delta>0$ 
such that  for any $N\in\N$,
$$ \mu\left(\left\|\frac{\tau_N}{N}- \mu(\tau)\right\|
\geq \eps\right)\leq C N^{-\delta}. $$
 $\tau$ satisfies {\em superpolynomial large deviation bounds}, 
if for each $w>0,$ $\eps>0$ there exist $C=C(\eps, w)$ 
such that  for any $N\in\N$,
$$ \mu\left(
\left\|\frac{\tau_N}{N}- \mu(\tau)\right\|
\geq \eps\right)\leq C N^{-w}. $$
\end{definition}

We will often use the following standard fact.
\begin{lemma}\label{lem:sob}
For each $r$, there is $w=w(r)$ such that functions $\Phi\in C^w(X\times Y)$ admit a decomposition
$\DS  \Phi(x,y)=\sum_{k=1}^\infty A_k(x) B_k(y)$, where $A_k\in C^r(X),$ $ B_k\in C^r(Y)$ 
and
\begin{equation}
\label{ProdDeco}
\sum_k \|A_k\|_{C^r(X)} \|B_k\|_{C^r(Y)} \leq C(r, w) \|\Phi\|_{C^w(X\times Y)}. 
\end{equation}
\end{lemma}
\begin{corollary}
\label{CrProductReduction}
Suppose that there are positive constants $K$ and $r$, such that
$$ \left|\iint A'(x) B'(y) A''(f^n x) B''(G_{\tau_n(x)} y) d\mu(x)\; d\nu(y)-\mu(A') \nu(B')\mu(A'') \nu(B'')\right|$$
\begin{equation}
\label{CorrProd}
\leq K \|A'\|_{C^r(X)}\|B'\|_{C^r(Y)}\|A''\|_{C^r(X)}\|B''\|_{C^r(Y)} \psi(n). 
\end{equation}
Then $F$ is mixing with rate $\psi.$
\end{corollary} 

\begin{proof}
Let 
$$\brrho_n(\Phi', \Phi''):=\zeta(\Phi' (\Phi''\circ F^n))-\zeta(\Phi')\zeta(\Phi''). $$
Decomposing $\Phi', \Phi''\in C^w$ as in \eqref{ProdDeco}, we get
$$\left| \brrho_n(\Phi', \Phi'')\right|=\left|\sum_{j, k} \brrho_n(A_j' B_j', A_k''B_k'')\right|
\leq K \psi(n) \sum_{j, k}\left( \|A_j'\|_r \|B_j'\|_r \|A_k''\|_r \|B_k''\|_r\right) $$
$\DS  \leq K \psi(n) \sum_{j }\left( \|A_j'\|_r \|B_j'\|_r \right) \sum_k \left(\|A_k''\|_r \|B_k''\|_r\right)
\leq K \psi(n) C^2 (r, w) \|\Phi'\|_w \|\Phi''\|_w. 
$
\end{proof}

\section{Mixing rates for mixing fibers}
\label{ScMix-MF}

\subsection{Double mixing}
\begin{theorem}
\label{ThDrift}
Suppose that $\mu(\tau)\neq 0$. 

(a) If $\tau$ satisfies exponential large deviation bounds  and
$f$ and $G$ are exponentially mixing, then $F$ is exponentially mixing.

(b) If $\tau$ satisfies polynomial large deviation bounds and
$f$ and $G$ are polynomially mixing, then $F$ is polynomially mixing.

(c) If $\tau$ satisfies superpolynomial large deviation bounds and
$f$ and $G$ are rapidly mixing, then $F$ is rapidly mixing.
\end{theorem}

\begin{proof}
(a) For $i=1,2$, let $\Phi_i(x,y)=A_i(x)B_i(y)$ be a $C^r$ function on $X\times Y$. Let $\rho(t):=\int_Y B_1(y)B_2(G_{t}y)d\nu(y)$. Since $G$ is exponentially mixing, there exist constants $C_1>0$ and $\kappa>0$ such that \begin{equation}\label{e-7.1-1}|\rho(t)-\nu(B_1)\nu(B_2)|\le C_1 \|B_1\|_{C^r} \|B_2\|_{C^r}e^{-\kappa 
 \|t \|} .\end{equation}
Taking $\varepsilon=\| \mu(\tau)\|/2$
 in the definition of exponential large deviation bounds,
we find that
there exist $C_0>0$ and $\delta>0$ such that $\mu(T_N)\le C_0e^{-\delta N}$, where
$$T_N:=\{x\in X:
\| \tau_N(x)-N\mu(\tau) \| \ge N \| \mu(\tau) \| /2 \}.$$

Now note that \begin{equation}\label{e-7}\int_{X\times Y} \Phi_1(x,y)\Phi_2(F^N(x,y))d(\mu\times\nu)=\int_X A_1(x)A_2(f^N(x))(\rho(\tau_N(x)))d\mu(x).\end{equation}
We rewrite the last integral as the sum of two integrals $\mathcal I_1+\mathcal I_2$, where
$$\mathcal I_1=\int_{T_N} A_1(x)A_2(f^N(x))(\rho(\tau_N(x)))d\mu(x),$$and$$\mathcal I_2=\int_{X\backslash T_N} A_1(x)A_2(f^N(x))(\rho(\tau_N(x)))d\mu(x).$$
 By exponential large deviation bounds,
$|\mathcal I_1|\le C_2\mu(T_N)\le C_3e^{-\delta N}$. For $\mathcal I_2$, since $f$ is exponentially mixing, it is enough to show that
$$\Delta:=\left|\mathcal I_2-(\nu(B_1)\nu(B_2))\int_{X\backslash T_N} A_1(x)A_2(f^N(x))d\mu(x)\right|$$ is exponentially small. Indeed, by (\ref{e-7.1-1}) 
$$
\Delta\le 
\left| \int_{X\backslash T_N} |A_1(x)||A_2(f^N(x))||\rho(\tau_N(x))-\nu(B_1)\nu(B_2)|d\mu(x)\right|$$
$$\le
C_4\|A_1\|_0\|A_2\|_0\|B_1\|_r\|B_2\|_r \cdot e^{-\kappa_1 N}\le C_4\|A_1\times B_1\|_r\|A_2\times B_2\|_r\cdot e^{-\kappa_1N}
$$
with $\kappa_1=\kappa/2$. This finishes the proof.
 The proofs of parts (b) and (c) are analogous to part (a). We will omit them.
\end{proof}

\begin{remark}
In part (b) above, if $\tau$ satisfies polynomial large deviation bounds with rate $N^{-\delta_1}$, and $f$, $G$ are polynomially mixing with rate $N^{-\delta_2}$ and $N^{-\delta_3}$ respectively, then $F$ is polynomially mixing with rate $N^{-\min\{\delta_1,\delta_2,\delta_3\}}$.
\end{remark}

\begin{remark}
Observe that the LLT was not needed
in Theorem \ref{ThDrift} and so the theorem remains valid if $\reals^d$ is replaced by 
an arbitrary Lie group, in which case $\tau_N$ means the product
$$ \tau_N(x) =\tau(f^{N-1} x) \dots \tau(fx) \tau(x). $$
\end{remark}

\begin{definition}\label{def:anti}
Assume that a cocycle $\tau$ 
 is such that $\frac{\tau_n-D_n}{L_n}$ converges as 
$n\to\infty$ to a non atomic distribution.
We say that $\tau$ satisfies the {\em anticoncentration inequality} if for {\bf every} unit cube $\cC\subset \reals^d$, 
$$
\mu\Big(\{x\in X\;:\; \tau_N(x)\in \cC\}\Big)\leq CL_N^{-d},
$$
for some global constant $C>0$.
\end{definition}

\begin{remark}
Note that by assumption there is a constant $R$ such that 
$$\mu(\|\tau_n\|\leq R L_n)\geq 0.5$$
so the power of $L_N$ in the anticoncentration inequality is optimal.  
\end{remark}

\begin{theorem} 
\label{ThAntiConMF}
Assume that for some $r\in \N$, $f$ is mixing with rate $\psi_f(N)=L_N^{-\alpha}$, for 
some $\alpha>0$ on $C^r,$  $\tau$ satisfies the anticoncentration inequality and $G$ is mixing with rate $\psi_G(\cdot)$ on $C^r$, where

\begin{equation}\label{eq:bd}
\int_{\reals^d}\psi_G(t)dt<+\infty.
\end{equation}
Then  $F$ is mixing with rate $\psi_F(N):=L_N^{-\min\{d,\alpha\}}$ on $C^w$ for some $w=w(r)\in \N$.
\end{theorem}

\begin{theorem}
\label{ThLLT-MF}
Assume that for some $r\in \N$, $f$ is mixing with rate $\psi_f(N)=L_N^{-\alpha}$, for 
some $\alpha>0$ on $C^r,$ $G$ is mixing with rate $\psi_G(\cdot)$ on $C^r$, 
$\tau$ satisfies the mixing LLT with zero drift.

(a)  Suppose $\tau$ satisfies the
 anticoncentration inequality. If $\psi_G(\cdot)$ satisifies \eqref{eq:bd} and 
\begin{equation}
\label{eq:zerofibermean}
\int \Phi_1(x,y) d\nu(y)\equiv 0,
\end{equation} then

\begin{equation}\label{TTInvGreenKubo}
\int \Phi_1(z) \Phi_2(F^N z) d\zeta(z)= 
\end{equation}
$$ \fp(0) L_N^{-d} \iiiint \Phi_1(x, y)\Phi_2(\brx, G_t y) d\mu(x) d\nu(y) d\mu(\brx) 
dt+o\left(L_N^{-d}\right). $$

(b) If $\psi_G(t)=\|t\|^{-\beta}$, for $\beta<d$, 
then  $F$ is mixing with rate $\psi_F(N):=L_N^{- \min\{\beta, \alpha\}}$ on $C^w$ for some $w=w(r)\in \N$.

(c) If 
$\min \{\alpha, d \}>\beta$ and for zero mean functions we have  
$$ \int B_1(y) B_2(G_t y) d\nu=q(B_1, B_2) \Psi(t)+o(\|t\|^{-\beta}) $$
where $q$ is a bounded bilinear form on $C^r(Y)$ and $\Psi$ is a homogeneous function
of degree $-\beta$, then
\begin{equation}\label{TTInvDiv}
\int \Phi_1(z) \Phi_2(F^N z) d\zeta(z)
=L_N^{-\beta} Q(\Phi_1, \Phi_2) 
\int_{\reals^d} \fp(t) \Psi(t) dt+o\left(L_N^{-\beta}\right)
\end{equation}
where 
$$ Q(\Phi_1, \Phi_2)=\int q(\Phi(x_1, \cdot), \Phi_2(x_2, \cdot)) d\mu(x_1) d\mu(x_2) . $$

\end{theorem}

\begin{remark}
In the case $d=1$, \eqref{TTInvGreenKubo} is proven in \cite{LB06} under a slightly more restrictive condition.
\end{remark}

\begin{remark}
We note that the integral in \eqref{TTInvDiv} converges. In fact, convergence near $0$ follows because
$\fp$ is bounded and $d > \beta$, while convergence near infinity follows since $\Psi$ is bounded outside of the unit sphere.
We also observe that for $\Phi_j(x,y)=A_j(x) B_j(y)$
\begin{equation}
\label{Q-q}
 Q(\Phi_1, \Phi_2)=\mu(A_1) \mu(A_2) q(B_1, B_2).
 \end{equation}
\end{remark}

\begin{proof}[Proof of Theorem \ref{ThAntiConMF}]
For $i=1,2$, let $\Phi_i(x,y)=A_i(x)\tilde{B}_i(y)$, where $A_i\in C^r(X)$ and $\tilde{B}_i\in C^r(Y)$.  Let $B_i=\tilde{B}_i-\nu(\tilde{B}_i)$.
Let $\rho(t):=\int_Y B_1(y)B_2(G_{t}y)d\nu(y)$. 
Note that \begin{multline}\label{e-77}\int_{X\times Y} \Phi_1(x,y)\Phi_2(F^N(x,y))d(\mu\times\nu)=
\int_X A_1(x)A_2(f^N(x))\cdot\rho(\tau_N(x))d\mu(x)+\\
\nu(\tilde{B}_1)\nu(\tilde{B}_2)\int_XA_1(x)A_2(f^N(x))d\mu(x).\end{multline}
Since $f$ is mixing with rate $L_N^{-\alpha}$ on $C^r$, the second summand is equal to $\mu(A_1)\mu(A_2)$ up to an error less than $C\|A_1\|_r\|A_2\|_r L_N^{-\alpha}$. 
It remains to estimate the first summand.

Let $\{\cC_i\}_{i=1}^{\infty}$ be a countable disjoint family of unit cubes in $\reals^d$ such that 
$\reals^d=\bigcup_{i}\cC_i$.  
Below we assume without the loss of generality 
that the function $\psi$ from \eqref{eq:bd} satisfies
\begin{equation}
\label{Sup-Inf}
\sup_{\cC_i} \psi(t) \leq K \inf_{\cC_i} \psi(t). 
\end{equation}
Indeed, given $t, \brt\in \cC_i$ we have
$$ \nu(B_1\cdot B_2\circ G_t)=\nu(B_1 \cdot\hB_2\circ G_\brt) $$
where $\hB_2=B_2\circ G_{t-\brt}.$ The last integral is smaller in absolute value than
$$ \psi(\brt) \|B_1\|_{C^r} \|\hB_2\|_{C^r}\leq 
K \psi(\brt) \|B_1\|_{C^r} \|B_2\|_{C^r}. $$
Thus decreasing $\psi$ if necessary we may assume that \eqref{Sup-Inf} holds.

Note first that since $\tau$ is bounded, we have
\begin{equation}
\label{LmCubeSum}
 \int_X A_1(x)A_2(f^N(x))\cdot\rho(\tau_N(x))d\mu(x)=
\sum_{i=1}^{\infty} \int_XA_1(x)A_2(f^N(x))\cdot\rho(\tau_N(x))\one_{\cC_i}(\tau_N(x))d\mu(x).
\end{equation}

Using that $G$ is mixing with rate $\psi_G$ on $C^r$,  \eqref{LmCubeSum} shows that

$$\left| \int_X A_1(x)A_2(f^N(x))\cdot\rho(\tau_N(x))d\mu(x)\right|\leq $$
$$C\|A_1\|_0\|A_2\|_0\|B_1\|_r\|B_2\|_r\sum_{i=1}^{\infty}[\sup_{t\in \cC_i}\psi_G(t)]\mu(\{x\in X\;:\; \tau_N(x)\in \cC_i\}).
$$
Together with the anticoncentration inequality, we have
{\small
\begin{equation}
\label{Cor-CubeSum}
\left| \int_X A_1(x)A_2(f^N(x))\cdot\rho(\tau_N(x))d\mu(x)\right|\leq C D\cdot\|A_1\|_0\|A_2\|_0\|B_1\|_r\|B_2\|_rL_N^{-d}\sum_{i=1}^{\infty}\sup_{t\in \cC_i}\psi_G(t).
\end{equation}}
Now by \eqref{Sup-Inf}
\begin{equation}
\label{Sup-Int}
\sum_{i=1}^{\infty}\sup_{t\in \cC_i}\psi_G(t)\leq C'\int_{\reals^d}\psi_G(t)dt<C''.
\end{equation}
Summarizing, we get 
$$
\int_XA_1(x)A_2(f^N(x))\cdot\rho(\tau_N(x))d\mu(x)\le C'''\|A_1\|_0\|A_2\|_0\|B_1\|_r\|B_2\|_rL_N^{-d}
$$
showing that $F$ is mixing with rate $L_N^{-\min\{d,\alpha\}}.$
\end{proof} 

\begin{proof}[Proof of Theorem \ref{ThLLT-MF}]
By the same argument in the proof of Theorem \ref{ThAntiConMF} we just need to estimate
$$ \int_XA_1(x)A_2(f^N(x))\cdot\rho(\tau_N(x))d\mu(x). $$
To prove part (a) note that due to \eqref{EqPortmanteau} for each fixed $i$,
$$ 
\lim_{N\to\infty} L_N^{d} \int_X A_1(x)A_2(f^N(x))\cdot\rho(\tau_N(x))\one_{\cC_i}(\tau_N(x))d\mu(x)=
\fp(0) \int_{\cC_i} \rho(t)  dt\;  \mu(A_1) \mu(A_2). $$
This together with the Dominated Convergence Theorem 
(note that in part (a) we assume the conditions of 
Theorem \ref{ThAntiConMF} whence \eqref{Cor-CubeSum} and \eqref{Sup-Int} apply) 
shows that
$$ \lim_{N\to\infty} L_N^{d} \int_X A_1(x)A_2(f^N(x))\cdot\rho(\tau_N(x))
d\mu(x)=
 \fp(0)
\mu(A_1) \mu(A_2) \int_{\reals^d} \rho(t) dt $$
proving \eqref{TTInvGreenKubo}.

To prove part (b), split
$\DS 
\int_XA_1(x)A_2(f^N(x))\cdot\rho(\tau_N(x))d\mu(x)=S_1+S_2,
$ where 
$$
S_1:=\int_{X} A_1(x)A_2(f^N(x))\cdot\rho(\tau_N(x))\one_{[-L_N,L_N]^d}(\tau_N(x))d\mu(x),
$$ 
and
$$
S_2:=\int_{X} A_1(x)A_2(f^N(x))\cdot\rho(\tau_N(x))\one_{\reals^d\setminus[-L_N,L_N]^d}(\tau_N(x))d\mu(x).
$$

To estimate $S_2$, notice that for $x$ as in $S_2$,  
$$\rho(\tau_N(x))\leq C\|B_1\|_r\|B_2\|_r \psi(\tau_N(x))\leq C_0\|B_1\|_r\|B_2\|_r \psi(L_N)\le C_0\|B_1\|_r\|B_2\|_rL_N^{-\beta}.$$  
Therefore 
$\DS 
S_2\leq C_0\|A_1\|_0\|A_2\|_0\|B_1\|_r\|B_2\|_rL_N^{- \beta}.
$

It remains to estimate $S_1$.
We trivially have 
\begin{align}\label{eq:aa}
\left|S_1\right|=&
\left|\int_{X} A_1(x)A_2(f^N(x))\cdot\rho(\tau_N(x))\one_{[-L_N,L_N]^d}(\tau_N(x))d\mu(x)\right|
\nonumber\\\le& \|A_1\|_0\|A_2\|_0 \int_{X}|\rho(\tau_N(x))|\one_{[-L_N,L_N]^d}(\tau_N(x))d\mu(x).
\end{align}
Cover $[-L_N,L_N]^d$ with (at most) $\Big([L_N]+1\Big)^d$ disjoint cubes $\{I_j\}$ of size $1$ 
centered at $t_j$, so that $I_j$'s are  translates of the cube $I_0$. 
By the mixing LLT for $z_n=t_j$ (notice that $\|t_j\|\leq  d L_N$ and so $t_j/L_N$ belongs to a compact set), and $A_0=A_1=1$, we get (for sufficiently large $N$),
$$
L_N^d\mu(\{x\in X\;:\; \tau_N(x)\in I_j\})< 2\fp^*{\rm Vol}(I_0)=2\fp^*
$$
where $\fp^*=\sup_t \fp.$  Therefore,
$$
\int_{X}|\rho(\tau_N(x))|\one_{[-L_N,L_N]^d}(\tau_N(x))d\mu(x)=\sum_j \int_{X}|\rho(\tau_N(x))|\one_{I_j}(\tau_N(x))d\mu(x)$$
$$\le
2 \fp^* L_N^{-d}\sum_j\sup_{t\in I_j}|\rho(t)|
\leq C  L_N^{-d} \int_{[-L_N, L_N]^d} \rho(t) dt
\leq C L_N^{-d} L_N^{d-\beta}=C L_N^{-\beta} ,
$$
 completing the proof of (b). 
 
To prove part (c), fix a small $\delta$ and split
$$ 
\int_XA_1(x)A_2(f^N(x))\cdot\rho(\tau_N(x))d\mu(x)=S_1+S_2+S_3
$$ where  
the integrand in $S_1$ is multiplied by $\one_{[-\delta L_N, \delta L_N]^d}(\tau_N(x))$, the integrand in $S_2$ is
multiplied by 
$$\one_{[-L_N/\delta, L_N/\delta]^d\setminus [-\delta L_N, \delta L_N]^d}
(\tau_N(x))$$
and the integrand in $S_3$ is multiplied by 
$\one_{\reals^d \setminus [-L_N/\delta, L_N/\delta]^d}(\tau_N(x)).$
Arguing as in the proof of part (b) we obtain that 
$S_3=O\left(\left(\frac{\delta}{L_N} \right)^{\beta}\right)$.
Since the integrand is bounded, we have
$S_1 = O\left(\left(\frac{\delta}{L_N} \right)^{d} \right)
=O\left(\left(\frac{\delta}{L_N} \right)^{\beta}\right)$.
To handle $S_2$ we divide the domain of integration into unit cubes $I_j.$ Let
$t_j$ to be the center of $I_j.$ Using the homogenuity of $\Psi$ we conclude from the mixing LLT that
$$ \int A_1(x) A_2(f^N x) \rho(\tau_N(x)) \one_{I_j}(\tau_N(x)) d\mu(x)$$
$$=
L_N^{-(d+\beta)} \mu(A_1) \mu(A_2) q(B_1, B_2) \;\fp\left(\frac{t_j}{L_N}\right)  \Psi\left(\frac{t_j}{L_N}\right) +o\left(L_N^{-(d+\beta)}\right) . $$
Summing over $j$ and using \eqref{Q-q} we obtain
$$ S_2=
L_N^{-\beta} Q(\Phi_1, \Phi_2) 
\int_{\cT_\delta} \fp(t) \Psi(t) dt+o\left(L_N^{-\beta}\right)  $$
where the domain of integration is
$\DS  \cT_\delta=\left[-\frac{1}{\delta}, \frac{1}{\delta}\right]^d\setminus [-\delta, \delta]^d. $
Combing our estimates for $S_1, S_2$ and $S_3$ we obtain
$$ \int \Phi_1(z) \Phi_2(F^n z) d\zeta(z)=
L_N^{-\beta} Q(\Phi_1, \Phi_2) 
\int_{\cT_\delta} \fp(t) \Psi(t) dt+o\left(L_N^{-\beta}\right) +
O\left(\left(\frac{\delta}{L_N}\right)^\beta\right). $$
Letting $\delta\to 0$ we obtain \eqref{TTInvDiv} for product observables, which 
by Lemma \ref{lem:sob} is sufficient to conclude the general case.
\end{proof}

\begin{remark}
\label{RmGKRelative}
Note that the fact that $\BAN=C^r$ was only used to decompose any $\Phi\in C^w(X\times Y)$ as
\begin{equation}
\label{Ban-Sum}
  \Phi(x,y)=\sum_n A_n(x) B_n(y),  \text{ where } 
    \sum_n \|A_n\|_{C^r} \|B_n\|_{C^r}<\infty. 
\end{equation}

Therefore the conclusions of 
Theorems \ref{ThAntiConMF}, \ref{ThLLT-MF} remain valid if \eqref{Eq2Mix} holds on arbitrary space
$\BAN$
provided that $\Phi_1, \Phi_2$ admit decomposition \eqref{Ban-Sum}. 
\end{remark}

\begin{remark}
The results of this section apply (with obvious modifications) to continuous time $T, T^{-1}$ systems of the
form
\begin{equation}
\label{TTInvFlows}
F^t(x,y)=(\phi^t (x), G_{\tau_t(x)} y) 
\end{equation}
where $\phi$ is a flow on $X$ and 
\begin{equation}
\label{TimeInt} 
\tau_t(x)=\int_0^t \tau(\phi^s (x)) ds.
\end{equation}
Note that due to the fact that
$\DS \zeta(H_1 (H_2 \circ F^{n+\delta}))=\zeta(H_1 ((H_2 \circ F^{\delta})\circ F^n))$
it is sufficient to control the correlation at integer times. Next $F^1$ is 
$T, T^{-1}$-transformation
corresponding to $f=\phi^1,$ $\tau=\tau_1.$ We note however, that in several case for time one maps
of the flow the LLT is unknown (or false) unless the observable is the time integral given by 
\eqref{TimeInt}. We refer the reader to \cite{DN-Flows} for the discussion of mixing LLT for continuous time
systems.
\end{remark}

\begin{example}
\label{TwoFlowEx}
(a) Let $g_t$ be an exponentially mixing Anosov flow on some manifold $M.$ Consider a continuous
$T, T^{-1}$ system $F_1^t$ with $X=Y=M$ and $\phi^t=G_t=g^t.$ Then Theorem \ref{ThLLT-MF}(a) shows that
for smooth zero mean observables
$$ \lim_{t\to\infty} \sqrt{t} \zeta(H_1 (H_2\circ F^t))=Q_1(H_1, H_2) $$
where $Q_1$ is given by \eqref{TTInvGreenKubo}.
Indeed, the condition \eqref{eq:zerofibermean} can be relaxed and 
the conclusion of Theorem \ref{ThLLT-MF}(a) holds for all zero mean
smooth observables assuming that $\alpha > d$ (in this example,
$\alpha$ is arbitrarily large and $d=1$).

(b) For any positive integer $k$, define indutively a continuous
$T, T^{-1}$ system $F_k^t$ with $X=M,$ $Y=M^k$, 
$\phi^t=g^t$ and $G_t=F_{k-1}^t$, where $F_1^t$ is 
the flow from the part (a). 
Then Theorem \ref{ThLLT-MF}(c) shows that
for smooth zero mean observables
$$ \lim_{t\to\infty} t^{ 2^{-k}} \zeta(H_1 (H_2\circ F^t))=Q_k(H_1, H_2) $$
where $Q_k$ is given in terms of $Q_{k-1}$ by \eqref{TTInvDiv}.
\end{example}

\subsection{Multiple mixing}
\label{ScMult-MF}

\begin{definition}
$G_t$ is  {\em mixing of order $s$ with rate $\psi$ on a space $\BAN$} if
$$ \left| \nu\left(\prod_{j=1}^s B_j(G_{t_j} y)\right) -\prod_{j=1}^s \nu(B_j) \right|\leq
C \psi(\delta(t_1, \dots t_s)) \prod_{j=1}^s \|B_j\|_\BAN $$ 
where
$$ \delta(t_1, \dots t_s)=\min_{i\neq j} \|t_i-t_j\|. $$
This definition extends to maps (such as to $f$ and $F$) in the 
natural way.
\end{definition}

\begin{theorem}
  If $\tau$ satisfies mixing LLT with zero drift
  and 
$f$ and $G$ are mixing of order $s$ with rate $t^{-\alpha}$ with $\alpha>d$, then
$F$ is mixing of order $s$ with rate $\psi_F(N)= L_N^{-d}.$ 
\end{theorem}
\begin{proof} For $i=1,\ldots, s$, let $\Phi_i(x,y)=A_i(x)B_i(y)$, where $A_i\in C^r(X)$ and $B_i\in C^r(Y)$. Let $\rho(t_1,t_2,\ldots, t_{s}):=\int_Y \prod_{i=1}^{s}B_i(G_{t_i}y)d\nu(y)$ (with $t_1=0$). 
We have
\begin{equation}
\label{eq77}
\int_{X\times Y} \prod_{i=1}^{s}
\Phi_i(F^{N_i}(x,y))d(\mu\times\nu)=
\int_X \prod_{i=1}^{s}A_i(f^{N_i}x)\cdot\rho(\tau_{N_1}(x),\ldots,\tau_{N_{s}}(x))d\mu(x)
\end{equation}
$$
\int_X \prod_{i=1}^{s}A_i(f^{N_i}x)\cdot\Big(\rho(\tau_{N_1}(x),\ldots,\tau_{N_{s}}(x))-\prod_{i=1}^{s}\nu(B_i)\Big)d\mu(x)\nonumber+
\prod_{i=1}^{s}\nu(B_i)\int_X \prod_{i=1}^{s}A_i(f^{N_i}x)d\mu(x).
$$

Note that since $f$ is mixing of order $s$ with rate $N^{-\alpha}$,
the last term above is equal to $\prod_{i=1}^s\mu(A_i)\nu(B_i)$ up to an error of size at most
$\DS {\rm O}\left(\prod_{i=1}^s\|A_i\|_r\min_{i\neq j}
|N_i-N_j|^{-\alpha}\right)$.
It is therefore enough to bound the first term.
Notice moreover that since $\tau$ is bounded and satisfies mixing LLT with zero drift,
we have
 $L_N\leq C'N$
(see Definition \ref{DfnMixLLT}).

Denote $\bar{N}:=\min_{i\neq j} |N_i-N_j|$.

Let $Z\subset X$ be defined by setting: $x\in Z$ iff $\min_{i\neq j}\|\tau_{N_i}(x)-\tau_{N_j}(x)\|\geq L_{\bar{N}}$. Using that $G$ is mixing of order $s$ with rate $\|t\|^{-\alpha}$, we get 
\begin{equation}\label{eq:z}
\int_Z \prod_{i=1}^{s}A_i(f^{N_i}x)\cdot\Big(\rho(\tau_{N_1}(x),\ldots,\tau_{N_{s}}(x))-\prod_{i=1}^{s}\nu(B_i)\Big)d\mu(x)\leq C\prod_{i=1}^s\|A_i\|_0\prod_{i=1}^s\|B_i\|_rL_{\bar{N}}^{-\alpha}.
\end{equation}
So it remains to estimate the above integral on $Z^c$. By definition, for every $x\in Z^{c}$,
there exists $i_x\neq j_x$ such that 
\begin{equation}
  \label{DefIJx}
  \|\tau_{N_{i_x}}(x)-\tau_{N_{j_x}}(x)\|=\min_{i\neq j}\|\tau_{N_i}(x)-\tau_{N_j}(x)\|\leq L_{\bar{N}}.
\end{equation}  
Let $Z_{ij}:=\{x\in Z^c\;:\; (i_x,j_x)=(i,j)\}$ (if there are several pairs satisfying \eqref{DefIJx} we take the smallest with respect to the
 lexicographic order).
Let $\{C_k\}_{k=1}^{\bar{M}}$ be a finite family of unit cubes centered at $\{c_k\}_{k=1}^{\bar{M}}$
in $\reals^d$ such that  $[-L_{\bar{N}},L_{\bar{N}}]^d=\bigcup_{k}C_k$. Then 
\begin{multline}
\label{eq:cs11}
\left|\int_{Z_{ij}} \prod_{ l=1}^s
A_l(f^{N_l}x)\cdot\Big(\rho(\tau_{N_1}(x),\ldots,\tau_{N_{s}}(x))-\prod_{i=1}^{s}\nu(B_i)\Big)d\mu(x)\right|=\\
\left|\sum_{k=1}^{\bar{M}} \int_{Z_{ij}} \prod_{ l=1}^s
A_l(f^{N_l}x)\cdot\Big(\rho(\tau_{N_1}(x),\ldots,\tau_{N_{s}}(x))-\prod_{i=1}^{s}\nu(B_i)\Big)\one_{C_{k}}(\tau_{N_i}(x)-\tau_{N_j}(x))d\mu(x)\right|.
\end{multline}

Using that $G$ is mixing of order $s$ with rate $\|t\|^{-\alpha}$, and
$$\min\{\sup_{C_k}\|t\|^{-\alpha}, 1\}\leq C\inf_{C_k}\|t\|^{-\alpha}$$
we get that  LHS of \eqref{eq:cs11} is bounded above by
{\small
\begin{equation}\label{eq:new1}
  C'\prod_{l=1}^s
  \left( \|A_l\|_0 \|B_l\|_r\right)
  \sum_{ k=1}^{\bar{M}}
  \left(\int_{C_k}
     \min\left\{\|t\|^{-\alpha}, 1\right\} dt\right)\;
       \mu(\{x\in X\;:\; \tau_{N_i}(x)-\tau_{N_j}(x)\in C_{k}\}).
\end{equation}}
Note that $\tau_{N_i}(x)-\tau_{N_j}(x)=\tau_{N_i-N_j}(f^{N_j}x)$. Hence, by the mixing LLT with $A_0=A_1=1$, $D_n\equiv 0$, we get (by preservation of measure)
$$
\mu(\{x\in X\;:\; \tau_{N_i}(x)-\tau_{N_j}(x)\in C_{j}\})
\leq 2L_{N_i-N_j}^{-d}
{\fp}(c_i / L_{\bar N})
<CL_{\bar{N}}^{-d}.
$$
Therefore, \eqref{eq:new1} (and hence also \eqref{eq:cs11}) is bounded above by (recall that $\alpha>d$)
$$
C''\prod_{ l=1}^s
  \left( \|A_l\|_0 \|B_l\|_r\right)
L_{\bar{N}}^{-d}.
$$
Summing over all $i,j$ and using \eqref{eq:z}, we get that the LHS of \eqref{eq77} is bounded by 
$$
C'''
\prod_{ l=1}^s
  \left( \|A_l\|_0 \|B_l\|_r\right)
L_{\bar{N}}^{-d}.
$$
This finishes the proof.
\end{proof}

\begin{theorem}
\label{ThEM+Drift}
  If $\tau$ has non zero drift and satisfies exponential large deviation bounds,
  and 
  $f$ and $G$ are exponentially mixing of order $s$
    then
$F$ is exponentially mixing of order $s.$ 
\end{theorem}
\begin{proof}
 For $i=1,2,\ldots, s$, let $\Phi_i(x,y)=A_i(x)B_i(y)$ be a $C^r$ function on $X\times Y$. Let $\rho(t_1,\ldots, t_s):=\int_Y \prod_{i=1}^sB_i(G_{t_i}y)d\nu(y)$ (with $t_1=0$). Since $G$ is exponentially mixing, there exist a constant $C_1>0$ and $\kappa>0$ such that \begin{equation}\label{e-7.1}|\rho(t_1,\ldots,t_s)-\prod_{i=1}^s\nu(B_i)|\le C_1 \|B_1\|_{C^r} \|B_2\|_{C^r}e^{-\kappa \delta(t_1,\ldots,t_s)} .\end{equation}
Fix $0=N_1\leq N_2\leq \ldots\leq N_s$. 
 We again use the decomposition \eqref{eq77}.
By exponential mixing of order $s$ of $f$, the second term in \eqref{eq77}
is exponentially close to  $\prod_{i=1}^{s}\nu(B_i)\prod_{i=1}^{s}\mu(A_i)$, and hence we only need to estimate the first term.

Let $T_{ij}:=\{x\in X:\; \|\tau_{N_i}(x)-\tau_{N_j}(x)-(N_i-N_j)
\mu(\tau)\|\geq (N_i-N_j)\|\mu(\tau)\|/2 \}$.
Let $\DS \bar{T}=\bigcup_{i\neq j}T_{ij}$. By exponential large deviation bounds (and preservation of measure), $\DS \mu(\bar{T})\leq s^2 \max_{ij} \mu(T_{ij})\leq Ce^{-\delta N}$. Therefore it is enough to bound the integral  of  the  first term in the RHS on $X\setminus \bar{T}$.
By exponential mixing of $G$,
 $$
\int_{X\setminus \bar{T}} \prod_{i=1}^{s}A_i(f^{N_i}x)\cdot\Big(\rho(\tau_{N_1}(x),\ldots,\tau_{N_{s}}(x))-\prod_{i=1}^{s}\nu(B_i)\Big)d\mu(x) 
$$$$
\leq C\prod_{i=1}^s\|A\|_0\prod_{i=1}^s\|B_i\|_r \min_{x\notin \bar{T}}e^{-\kappa \delta(\tau_{N_1}(x), \ldots,\tau_{N_s}(x))}. $$
By the definition of $\bar{T}$,
 $\DS
\delta(\tau_{N_1}(x), \ldots,\tau_{N_s}(x))\geq  \frac{\|\mu(\tau)\|}{2}\min_{i\neq j}|N_i-N_j|
$
completing the proof.
\end{proof}

Let $n_1 \leq n_2\leq \dots \leq n_s$ be a $s$ tuple.
A partition $\fP=\cP_1\cup \cP_2\cup\dots \cup \cP_{k}$ of a set 
$\{n_1, n_2, \dots n_{s}\}$ (where an item may be listed more than once)
is called {\it social} if for each $j\in \{1,\dots, 
k\}$, 
$\Card(\cP_j)>1.$ 
An element $n_j$ is called {\it forward free (backward free)} for partition $\fP$ if it is
the smallest (respectively, the largest) in its atom. We call $n_j$ forward (or backward) fixed
if it is not forward (backward) free. We let $F^\pm$ to be the set of all forward (or backward) fixed 
elements.
Let 
$$\kappa^{\pm}(\fP)=
\prod_{n_j\in F^{\pm}} L_{n_j-n_{j-1}}, \quad
$$
For $\fP=(P_1,\ldots,P_k)$, let $(n_{i_\ell})_{\ell=1}^k$ be the collection of forward free elements, i.e. $n_{i_\ell}$ is the smallest element of $P_\ell$. Analogously we define $(n_{j_\ell})_{\ell=1}^k$ to be the collection of backward free elements. Notice that we have the following formula for $\kappa^{\pm}(\fP)$:
\begin{equation}\label{eq:ff}
\kappa^+(\fP)=\Big(\prod_{j=1}^s L_{n_j-n_{j-1}}\Big)\cdot \Big(\prod_{\ell=1}^k L_{n_{i_\ell}-n_{i_\ell-1}}\Big)^{-1},
\end{equation}
with $n_0=0$ and analogously
\begin{equation}\label{eq:ff2}
\kappa^-(\fP)=\Big(\prod_{j=1}^s L_{n_j-n_{j-1}}\Big)\cdot \Big(\prod_{\ell=1}^k L_{n_{j_\ell+1}-n_{j_\ell}}\Big)^{-1},
\end{equation}
with $n_{s+1}:=n_1+n_s$.

We have the following 

\begin{definition}
  $\tau$ satisfies {\em anticoncentration large deviation bound of order $s$} if
  there exist a constant $K$ and a decreasing function $\Theta$ such that
  $\DS \int_1^\infty \Theta(r) r^{d}<\infty$,
and
for any unit cubes $C_1, C_2, \dots, C_s$ centered at $c_1, c_2, \dots c_s$ 
$$ \mu\left(x: \tau_{n_j}\in C_j \text{ for } j=1,\dots, s\right)
\leq K \left(\prod_{j=1}^s L_{n_{j}-n_{j-1}}^{-d}\right) 
\Theta\left(\max_{j} \frac{\|c_j-c_{j-1}\|}{L_{n_j-n_{j-1}}}\right) $$
\end{definition}  

\begin{remark}
For $s=2$ anticoncentration large deviation bounds were considered in \cite{DN-Ren}. 
\end{remark}

\begin{theorem}\label{thm:social}
If $\tau$ satisfies anticoncentration large deviation bounds  of order $s$ and 
$f$ and $G$ are exponentially mixing of order $s$, then
\begin{equation}
\label{MultMixProd}
\left| \int \left(\prod_{j=1}^s H_j(F^{n_j} z) \right) d\zeta(z)-\prod_{j=1}^s \zeta(H_j)\right|\leq 
C \prod_{j=1}^s \|H_j\|_{C^r} \left( \min_\fP \kappa (\fP) \right)^{-d}
\end{equation}
where 
$$ \kappa(\fP)=
\max
\left\{\kappa^+(\fP), \kappa^-(\fP)\right\}. $$
and the minimum 
in \eqref{MultMixProd}
is taken over all social partitions of $\{n_1,\dots n_s\}.$
\end{theorem} 
We first recall the following result, which simplifies our analysis.

\begin{lemma}(\cite{BG19})\label{lem:as} If $G$ is exponentially mixing of order $s$, then for some $\eta>0$
\begin{equation}
\label{Eq:as}
 \left| \nu\left(\prod_{j=1}^s B_j(G_{t_j} y)\right) -\prod_{j=1}^s \nu(B_j) \right|\leq
Ce^{-\eta\Delta(t_1, \dots t_s)} \prod_{j=1}^s \|B_j\|_\BAN ,
\end{equation}
where
$$ \Delta(t_1, \dots t_s)=\max_{j}\min_{i\neq j} \|t_i-t_j\|. $$
\end{lemma}
With the above lemma, we prove Theorem \ref{thm:social}
\begin{proof}[Proof of Theorem \ref{thm:social}]By Lemma \ref{lem:sob} it is enough to show the statement for  $H_j=A_j\times B_j\in C^r(M)$.  Let 
$$
\rho(t_1,\ldots, t_{s}):=\nu\left(\prod_{j=1}^{s} B_j(G_{t_j} y)\right) -
\prod_{j=1}^{s} \nu(B_j).
$$
Then
$$
 \int \left(\prod_{j=1}^{s} H_j(F^{n_j} z) \right) d\zeta(z)=$$
\begin{equation}
\label{eq:socialdec}
\int \left(\prod_{j=1}^{s} A_j(f^{n_j}x)\right)\rho\Big(\tau_{n_1}(x),\ldots,\tau_{n_{s}}(x)\Big)d\mu(x)+   \left(\prod_{j=1}^{s} \nu(B_j)\right) \mu\left(\prod_{j=1}^{s} A_j(f^{n_j} x)\right).
\end{equation}
Since $f$ is exponentially mixing 
of order $s$, 
\begin{equation}\label{eq:1b}
\left|\mu\left(\prod_{j=1}^{s} A_j(f^{n_j} x)\right)-\prod_{j=1}^{s} \mu(A_j)\right|\leq C \prod_{j=1}^{s}\|A_j\|_r e^{-\eta \Delta},
\end{equation}
where $\Delta=\Delta(n_1,\ldots,n_s)$.

Let $\cP$ be the following partition of $n_1<\ldots <n_{s}$. Let $i_1\in \{2,\ldots s-1\}$ 
be the smallest index $i$ such that $|n_{i}-n_{i-1}|>\Delta$. Then the first atom of $\cP$ is $\{n_0,\ldots, n_{i_1-1}\}$. Notice that $|n_{i_1}-n_{i_1+1}|
 \leq 
\Delta$ by the definition of $\Delta$. Now recursively, let $i_{k+1}\in \{i_k+1,\ldots, s\}$ be the smallest 
index $i$ such that $|n_{i}-n_{i-1}|>\Delta$.
Then the $(k+1)$-th atom of $\cP$ is $\{n_{i_k},\ldots, n_{i_{k+1}-1}\}$. We continue until we partition all of $n_1<\ldots n_{s}$. Then by the definition of $\Delta$, every atom of $\cP$ has at least two elements, and so $\cP$ is social. Moreover, all elements  in one atom
are at distance at most $s\Delta$ (since the number of elements is $\leq s$). Using that $\tau$ is bounded (and so  $|L_n|<Cn$) together with \eqref{eq:ff}
and \eqref{eq:ff2}, we conclude
$$
\min\{\kappa^+(\cP)^{-d},\kappa^+(\cP)^{-d}\}\geq \Big(\prod_{j=1}^s L_{n_j-n_{j-1}}\Big)^{-d}\gg [s\Delta]^{-sd}\geq C \Delta^{-sd}\geq Ce^{-\eta \Delta}.
$$

Combining this estimate with \eqref{eq:1b} we find that the 
second term in \eqref{eq:socialdec} 
equals $\DS \prod_{j=1}^s \zeta(H_j)$ up to an error which is
bounded by the
RHS of \eqref{MultMixProd}.
It remains 
to show that 
$$
\left|\int \left(\prod_{j=1}^{s} A_j(f^{n_j}x)\right)\rho\Big(\tau_{n_1}(x),\ldots,\tau_{n_{s}}(x)\Big)d\mu(x)\right|\leq C \prod_{j=1}^{s} \|A_j\times B_j\|_{C^r} 
\left(\min_{\fP} \kappa(\fP)\right)^{-d},
$$
which will follow by showing that 
$$
\int\Big|\rho(\tau_{n_1}(x),\ldots,\tau_{n_{s}}(x))\Big|d\mu(x)\leq C \prod_{j=1}^{s} \|B_j\|_{C^r} 
 \left(\min_{\fP} \kappa(\fP)\right)^{-d}.
$$
Let $C_0:=\prod_{j=1}^{s} \|B_j\|_{C^r}$ and let 
$D_m:=\Big\{x\;:\; |\rho(\tau_{n_1}(x),\ldots,\tau_{n_{s}}(x)) |\in [C_02^{-m},C_0 2^{-m+1})\Big\}$. Then 
\begin{equation}\label{int:a}
\int\Big|\rho(\tau_{n_1}(x),\ldots,\tau_{n_{s}}(x))\Big|d\mu(x)\leq 2C_0\sum_{m\geq 0}\frac{1}{2^m}\mu(D_m).
\end{equation}
We will estimate the measure of $D_m$. Note that by Lemma \ref{lem:as}, for some $C_\eta\in \N$,
$$
D_m\subset A_m:=\{x\;:\; \Delta(\tau_{n_1}(x),\ldots,\tau_{n_{s}}(x))\leq C_\eta m\}.
$$
We will therefore give an upper bound on the measure of $A_m$. By the definition of $\Delta$ it follows that there exists a social partition $\fP=(P_1,\ldots, P_k)$ of $n_1<n_2<\ldots<n_s$ such that for any atom of $\fP$ and any two $n_{i},n_j$ in the same atom we have 
\begin{equation}\label{eq:tau}
|\tau_{n_i}(x)-\tau_{n_j}(x)|<C_\eta s m.
\end{equation}
Let $A_{m,\fP}\subset A_m$ be the set of $x$ for which $\fP$ is 
 social partition of 
$n_1<n_2<\ldots<n_s$ satisfying \eqref{eq:tau}. Then 
$$
A_m=\bigcup_{\fP\text{ social}}A_{m,\fP},
$$
and so we will estimate the measure of $A_{m,\fP}$.

Let $\{\tilde{C}_j\}$ be a disjoint cover of $\reals^d$ by cubes of side length $C_\eta s \cdot m$ centered and $\tilde{c}_j$. Note that by the anticoncentration large deviation bounds of order $s$ (decomposing $\tilde{C}_j$ into unit cubes), 
\begin{equation}\label{eq:sad}
\mu\left(x: \tau_{n_j}(x)\in \tilde{C}_j \text{ for } j=1,\dots, s\right)
\leq K'(sm)^{sd} \left(\prod_{j=1}^s L_{n_{j}-n_{j-1}}^{-d}\right) 
\Theta\left(\max_{j} \frac{\|\tilde{c}_j-\tilde{c}_{j-1}\|}{m\cdot L_{n_j-n_{j-1}}}\right).
\end{equation}
It follows by the definition of $\fP$ and \eqref{eq:tau} that 
all the $\{\tau_{n_j}(x)\}_{n_j\in P_\ell}$ belong to one cube $\tilde{C}_{r_\ell}$.  Below, we use the notation $\tau_{P_\ell}(x)\in C_{r_\ell}$ which means that for every $n_j\in P_\ell$, $\tau_{n_j}(x)\in  \tC_{r_\ell}$.
Therefore, we have
$$
\mu(A_{m, \fP})\leq \sum_{r_1,\ldots, r_k}\mu(\{x\;:\; \tau_{P_\ell}(x)\in  \tC_{r_\ell},  \ell \leq k \}. 
$$
Let $n_{i_\ell}$ (and $n_{j_\ell}$) be the smallest (the largest) element of $P_\ell$, $\ell\leq k$. Below we will argue with $(n_{i_\ell})$ (analogous reasoning can be done for $(n_{j_\ell})$). Let $u(\ell)$ be such that $n_{i_\ell-1}\in P_{u(\ell)}$.
By \eqref{eq:sad}, monotonicity of $\Theta$ and the above
discussion (using that $n_{i_\ell}$ and $n_{i_\ell-1}$ are in different atoms), we obtain
$$
\mu(\{x\;:\; \tau_{P_\ell}(x)\in C_{r_\ell},  \ell \leq k \}\leq K'm^{sd}
\left(\prod_{j=1}^s L_{n_{j}-n_{j-1}}^{-d}\right)\Theta\left(\max_{\ell\leq k} \frac{\|\tilde{c}_{r_\ell}-\tilde{c}_{r_{u(\ell)}}\|}{m\cdot L_{n_{i_\ell}-n_{i_\ell-1}}}\right).
$$
Therefore 
$$
\mu(A_{m,\fP})\leq K'm^{sd}\left(\prod_{j=1}^s L_{n_{j}-n_{j-1}}^{-d}\right)\sum_{r_1,\ldots r_k}\Theta\left(\max_{\ell\leq k} \frac{\|\tilde{c}_{r_\ell}-\tilde{c}_{r_{u(\ell)}}\|}{m\cdot L_{n_{i_\ell}-n_{i_\ell-1}}}\right).
$$
Note that 
$$
\sum_{r_1,\ldots r_k}\Theta\left(\max_{\ell\leq k} \frac{\|\tilde{c}_{r_\ell}-\tilde{c}_{r_{u(\ell)}}\|}{m\cdot L_{n_{i_\ell}-n_{i_\ell-1}}}\right)\leq$$$$
 \sum_{\ell}\Theta(\ell)\cdot\Big|\{(r_1,\ldots, r_k)\;:\; \|\tilde{c}_{r_\ell}-\tilde{c}_{r_{u(\ell)}}\|\leq \ell\cdot m\cdot L_{n_{i_\ell}-n_{i_\ell-1}}\text{ for every } \ell\leq k \}\Big|
$$$$
\leq \sum_{\ell}\Theta(\ell)\ell^d\cdot m^d\cdot \Big(\prod_{\ell\leq k}L_{n_{i_\ell}-n_{i_\ell-1}}\Big)^d.
$$
Therefore, by the decay assumptions on $\Theta$ and \eqref{eq:ff},
$$
\mu(A_{m,\fP})\leq K'm^{sd+d}\left(\prod_{j=1}^s L_{n_{j}-n_{j-1}}^{-d}\right)\cdot\Big(\prod_{\ell\leq k}L_{n_{i_\ell}-n_{i_\ell-1}}\Big)^d=K'm^{sd+d}\kappa^{+}(\fP)^{-d}.
$$
Analogously we have that 
$$
\mu(A_{m,\fP})\leq K'm^{sd+d}\kappa^+(\fP)^{-d}.
$$
Therefore, 
$$
\mu(A_{m,\fP})\leq K'm^{sd+d}\kappa(\fP)^{-d}.
$$
Using that $A_m=\bigcup_{\fP}A_{m,\fP}$, we get,
$$
\mu(A_{m})\leq K'C_sm^{sd+d}(\min_\fP\kappa(\fP))^{-d},
$$
for some constant $C_s>0$. Summarizing, by \eqref{int:a} (since $D_m\subset A_m$), we get
$$
\int\Big|\rho(\tau_{n_1}(x),\ldots,\tau_{n_{s}(x)})\Big|d\mu(x)\leq 2K'C_s\prod_{j=1}^{s} \|B_j\|_{C^r} (\min_\fP\kappa(\fP))^{-d}\sum_{m\geq 0}2^{-m}m^{sd+d}\leq$$$$ C_{s,d}\prod_{j=1}^{s} \|B_j\|_{C^r}(\min_\fP\kappa(\fP))^{-d}.
$$
This finishes the proof.
\end{proof}

\section{Central Limit Theorem}
\label{ScCLT}

Let $H(x,y)$ be a $C^r$ function
not cohomologous to a constant function. Let $\DS \Sigma_N(H):=\sum_{n=0}^{N-1} H(F^n(x,y))$. 
Assume that $\zeta(H)=0$. Let $Z=X\times Y$.

\begin{theorem}
\label{ThMultMixCLT}
Suppose that $F$ satisfies \eqref{MultMixProd} and $\DS \sum_{n=1}^\infty L_n^{-d}$ 
converges.
Then $\DS \frac{\Sigma_N(H)}{\sqrt N}$ converges as $N\to \infty$ to the normal distribution with zero mean and variance $\sigma^2$ given by formula \eqref{GKSigma} below. 
\end{theorem}

 \begin{corollary}
If $F$ satisfies either the assumptions of Theorem \ref{ThEM+Drift}
or the assumptions of Theorem \ref{thm:social} with 
$L_N\geq c \sqrt{N}$ and $d\geq 3$,  
then $F$ satisfies the CLT.
\end{corollary}

\begin{proof}
In the case of Theorem \ref{ThEM+Drift}, this follows from the CLT for exponentially mixing systems
(\cite{Ch95, BG19}).
In the case of Theorem \ref{thm:social}, the result follows from Theorem \ref{ThMultMixCLT}.
\end{proof}

\begin{proof}[Proof of Theorem \ref{ThMultMixCLT}]
By \eqref{MultMixProd} with $n_1=0,n_2=n$
\begin{equation}
\label{GKSigma}
\sigma^2:=\sum_{n=-\infty}^\infty \zeta( H (H\circ F^n))
\end{equation}
exists and is finite.
Hence
$$
\zeta \left(\frac{\Sigma^2_N(H)}{N}\right)=\frac{1}{N}\sum_{1\le i,j\le N}
 \zeta((H\circ F^i)(H\circ F^j))
=$$$$\sum_{k=-N+1}^{N-1}\frac{N-|k|}{N} \zeta( H (H\circ F^k))
\to \sum_{n=-\infty}^\infty \zeta\left( H\left( H\circ F^n\right)\right). 
$$

To finish our proof, we need to estimate the asymptotics of moments 
$\zeta (\Sigma^m_N(H))$, for any $m\ge 3$.  
Denote $$\Omega(k_1,\dots, k_m)=\int_Z \left(\prod_{i=1}^m H(F^{k_i}z) \right) d\zeta(z)$$
so that
\begin{equation}\label{s:clt}
\zeta (\Sigma^m_N(H))=\sum_{k_1,\dots, k_m =1}^N
\Omega(k_1,\dots, k_m).
\end{equation}

For the vector $(k_1,...,k_m)$ we associate another vector $(n_1,...,n_m)$ which
is the permutation of the elements of $(k_1,...,k_m)$ in increasing order, that is
$n_1 \leq n_2 \leq ... \leq n_m$
Noting that $\Omega$ is symmetric, we have 
$\Omega(k_1,\dots, k_m) = \Omega(n_1,\dots, n_m)$.
We rewrite the above sum into two terms as $I_1+I_2$, where 
$I_1$ is the sum of terms, whose social partition 
minimizing the RHS of \eqref{MultMixProd}
 is not pairing (i.e. at least one atom contains more than two elements), 
and $I_2$ is the sum of terms, whose 
corresponding social partition is pairing.
(If there are more than one partition minimizing 
$\kappa$ at least one
of which is not pairing then we put the corresponding term into $I_1.$)

We need two auxiliary estimates.
Let $Q = \{Q_1,...,Q_r\}$ 
be a fixed social partition of the set $\{ 1,2,...,m \}$.
We say that $\mathcal Q(n_1,...,n_m) = Q$
if the partition $\fP$ minimizing the RHS of \eqref{MultMixProd}
for the given numbers $n_1,...,n_m$ is of the form
$\fP = \{ P_1,...,P_r \}$ with 
$\{ i: n_i \in P_k \} = Q_k$ for all $k=1,...,r$. Next we write
$$I_Q = 
\sum_{k_1,\dots, k_m : \mathcal Q(n_1,...,n_m) = Q}
\Omega (n_1,...,n_m).
$$

\begin{lemma}\label{le:odd}
(a) $I_Q={\rm O}\left(N^{r}\right).$ 

(b) If $Q=Q_1\cup\cdots\cup Q_r$ is not pairing, then the sum $I_Q={\rm O}\left(N^{(m-1)/2}\right).$
\end{lemma}
\begin{proof}
Since $1/\kappa_Q(n_1, \dots n_m)\leq 1/\kappa_Q^+(n_1, \dots n_m)$, 
by \eqref{MultMixProd}
it suffices to estimate
\begin{equation}
\label{ForwardSum}
 \sum_{n_1, \dots n_m} \frac{1}{(\kappa_Q^+(n_1, \dots n_m))^d} . 
\end{equation} 
Let $n_1'<n_2'<
\dots <n_r'$  be the forward free elements 
among $\{n_1, \dots n_m\}$ and $n_1'', \dots n''_{m-r}$ be
the forward fixed elements. For each fixed element 
$n''_j$, 
let $\brn_j$ be the previous
element in $\{n_1, \dots n_m\}$.
Rewrite \eqref{ForwardSum} as 
\begin{equation}
\label{ForwardDSum}
 \sum_{n_1', \dots n_r'} 
\left[ \sum_{n_1'', \dots n_{m-r}''} \left(\frac{1}{\prod_{j=1}^{m-r} L_{n_j''-\brn_j}} \right)^d \right]. 
\end{equation}
Since $L_n^{-d}$ is summable, the inner sum is uniformly bounded, so that
\eqref{ForwardDSum} is bounded by $N^r.$  This proves (a).

(b) follows from (a) because if
 $Q$ is not pairing, then $r< 
\lfloor
m/2 \rfloor.$ 
\end{proof}

Since there are finitely many partitions of $\{1, \dots m\},$ Lemma \ref{le:odd} implies that $|I_1|$ is bounded above by $O\left(N^{(m-1)/2}\right)$. 
In particular, for odd $m$, $\zeta (\Sigma^m_N(H))={\rm O}(N^{(m-1)/2})$.

Now let $m$ be even, 
and $Q$ be a pairing, that is $Q = \{ Q_1,...,Q_{m/2} \}$
with all atoms $Q_k$ containing exactly two numbers.
By forward (backward) step we mean $n_{j}-n_{j-1}$ where $n_j$ is forward (backward) fixed
 in the partition $\mathcal Q(n_1, \dots n_m)$. Let $\Gamma_Q(n_1, \dots n_m)$ be a largest among all forward and backward steps
in the partition $Q$
and let $\Gamma(n_1, \dots n_m) = \Gamma_{\mathcal Q(n_1, \dots n_m)}(n_1, \dots n_m)$.

\begin{lemma}\label{le:even}
For any $ \epsilon>0$, there exists $ M>0$, such that 
$$\left|\sum_{k_1,\dots,k_m:\Gamma(n_1,\dots,n_m)>M}\Omega(n_1,\dots,n_m)\right|\le N^{m/2}\epsilon.$$
\end{lemma}
\begin{proof}

It is enough to prove the lemma for $\Gamma$ replaced by $\Gamma^+$ and also for $\Gamma$ replaced
by $\Gamma^-$, where
$\Gamma^+$ is a largest among all forward steps and $\Gamma^-$ is a largest among all forward steps.
We only consider $\Gamma^+$ as $\Gamma^-$ is similar.
The proof for $\Gamma^+$ proceed in the same way as the proof of Lemma \ref{le:odd} except 
we estimate the inner sum in \eqref{ForwardDSum} by
\begin{equation}
\label{ImprSum}
C \left(\sum_{n=1}^\infty L_n^{-d} \right)^{m-r-1}
\left(\sum_{n=M}^\infty L_n^{-d} \right).
\end{equation}
Indeed there are $m-r$ factors in the inner sum in \eqref{ForwardDSum}, and by our assumptions one of them
should be greater than $M$ 
As the second factor can be made as small as we wish by taking $M$ large 
and since $r = m/2$,
the result follows.
\end{proof}

\begin{lemma}
\label{lem:pairing}
Let $Q$ be a pairing which is different from 
\begin{equation}
\label{NaturalPairing}
 \brQ:=\left[(12), (34), \dots, ((m-1)\; m)\right].
\end{equation}
Then the number of
$m$-tuples $(k_1,...,k_m)$ with $\Gamma_{ \brQ}(n_1, n_2,\dots, n_m)<{L}$ is ${\rm O}\left(N^{(m/2)-1}\right)$,
where the implicit constant depends on {$L$}.
\end{lemma}

\begin{proof}
We claim that if $Q\neq \brQ$ then the sets of forward fixed and backward fixed edges 
are different.
If follows that if both $\Gamma_{\brQ}^+(n_1, \dots, n_m)<M$ and
$\Gamma_{\brQ}^-(n_1, \dots, n_m)<M$,
then there are at least $m/2+1$ edges which
are shorter that $M.$ The number of such tuples is $O(N^{(m/2)-1})$ and the result follows.

It remains to prove the claim. That is, we show that if the sets of forward fixed and backward
fixed edges are the same, {then} $Q=\brQ.$
We proceed by induction. If $m=0$ or $2$  then there are no pairings different 
from $\brQ.$ Suppose $m>2.$ Then $(n_{m-1}, n_m)$ is forward fixed, so it should be backward
fixed but this is only possible if $(m-1)$ is paired to $m.$ Likewise $(n_1, n_2)$ is backward fixed,
hence it is forward fixed. But this is only possible if 1 is paired to 2. Removing $1,2, (m-1)$ and $m$
from $Q$ we obtain a partition of $m-4$ elements for which the set of forward fixed and backward fixed
edges coincide. By induction $3$ is paired to $4$, $5$ to $6$,\dots, $(m-3)$ to $(m-2).$
The proof is complete.
\end{proof}

{
By the above lemmas, it suffices to consider indeces $k_1,...,k_m$
so that 
\begin{equation}
\label{eq:wellsep} \forall i=1,...,m/2:\,
M_i := n_{2i} - n_{2i-1} \leq M \text{ and }\forall i=1,...,m/2-1:\, 
n_{2i+1} - n_{2i} > L
\end{equation} for some large $M$ and $L=L(M)$.
Indeed, by choosing $M = M(\eps)$ and $N> N_0$, $N_0 = N_0(L)$, the above 
lemmas give that the contribution of other terms is $< \eps N^{m/2}$. 
Now we choose $L$ so that 
for any fixed $M_1,...,M_{m/2}$ (finitely many choices), the RHS of 
\eqref{MultMixProd} with $s = m/2$ and $H_j = H (H \circ T^{M_j})$ is less than 
$\eps$. We conclude
$$
\left|
\zeta(\Sigma_N^m) - \sum_{k_1,...,k_m \text{ satisfying \eqref{eq:wellsep}}}
\prod_{i=1}^{m/2}
\left(\int_Z\left(H(H \circ T^{M_i})\right)d\zeta(z)\right) \right| \leq 2 \eps N^{m/2}$$
Let us write $A_\ell = \int_Z\left(H(H \circ T^{\ell})\right)d\zeta(z)$.
Now we claim
$$
\sum_{k_1,...,k_m \text{ satisfying \eqref{eq:wellsep}}}
\prod_{i=1}^{m/2}
A_{M_i}
= (m-1)!! N^{m/2}(1 + o(1)) \left[ \sum_{\ell = 0}^{M} \left( 
A_{\ell}
(1 + \one_{\ell > 0})
\right) \right]^{m/2}.
$$
To prove the claim, first note that 
$$ \sum_{M_1,...,M_{m/2} = 0}^M A_{M_1} ... A_{M_{m/2}} = \left(
\sum_{\ell=0}^M A_\ell \right)^{m/2}.
$$
Now it remains to count the number of tuples $(k_1,...,k_m)$ corresponding to 
the values $M_1,...,M_{m/2}$. 
Assume for example that $M_i>0$ for all $i$.
To count the number of possibilities, we first fix a pairing of indeces $1,..,m$
which can be done in $(m-1)!!$ different ways. Then we have $\approx N^{m/2}$ choices to prescribe
exactly one element of each pair. Let us say these values are $s_1 < s_2< ... <s_{m/2}$.
Except for a $o(N^{m/2})$ of these choices, we have $s_i - s_{i-1} > 2M+L$ and so for each remaining
index $k_j$ we have two choices: if it is paired to $s_i$, then either $k_j = s_i - M_i$ or $k_j = s_i + M_i$.
Thus the total number of choices is $(m-1)!!2^{m/2}N^{m/2}(1+o(1))$ which verifies the claim 
for the case $M_i>0$ for all $i$. If $M_i = 0$ for some $i$, then we only have one choice for the corresponding 
$k_j$ and so we lose a factor of $2$. The claim follows.

To finish the proof, notice that 
$$
\sum_{\ell = 0}^{M}  
A_{\ell}
(1 + \one_{\ell > 0}) = \sum_{\ell = -M}^M \zeta(H(H\circ F^\ell)) \to \sigma^2  \text{ as } M \to \infty.
$$
Thus we have verified}
$$\zeta(\Sigma^m_N(H))=\begin{cases} {\rm o}(N^{m/2}),&\text{$m$ is odd},\\(m-1)!!N^{m/2}\sigma^{m}+
{\rm o}(N^{m/2}),
&\text{$m$ is even.}\end{cases}
$$
completing the proof of the theorem.
\end{proof}

\begin{remark}
The asymptotic variance given by \eqref{GKSigma} is typically non-zero. In particular, if either the drift is non zero, or
$d\geq 5$, then a direct calculation shows that
$$ \lim_{N\to \infty} \zeta(\Sigma_N^2)-N \sigma^2=-\sum_{n=-\infty}^{ \infty} 
n {\zeta( H (H\circ F^n))}  $$
(the convergence of the right hand side follows from the assumptions imposed above). 
Thus if $\sigma^2=0$ then $\zeta(\Sigma_N^2)$ is bounded so by 
$L^2$--Gotshalk-Hedlund Theorem $H$ is an $L_2$ coboundary.
It is an open question if the same conclusion holds if $\mu(\tau)=0$ and $d$ is $3$ or $4$.
However, by assumption, $f$ is exponentially mixing, so if $H$ does not depend on $y$ then
$\sigma^2>0$ unless $H$ is an $L^2$ coboundary. Thus in many (possibly all) cases $\sigma^2$ is
a positive semidefinite quadratic form which is not identically equal to zero, and so its null set is a 
a linear subspace of positive (or infinite) codimension.
\end{remark}

\section{Mixing rates for ergodic fibers}
\label{ScMix-EF}
\subsection{Results}

\begin{definition}
We say that $(f, \tau)$ satisfies a {\em mixing averaged Edgeworth expansion of order $r$}
if there   are constants $k_1, k_2$ 
and a sequence $\delta_N \to 0$ so that
for any function $\phi = \phi_N \in C^{k_2}(\mathbb R^d,\reals)$ supported on the box $J =J_N$, the expression

$$ \cI_{A_1, A_2, \phi}(N):=
\mu(A_1(x) A_2(f^N x) \phi(\tau_N(x))) $$
satisfies
$$  
\left| \cI_{A_1, A_2, \phi}(N)-
N^{-d/2} \int_{s \in \reals^d} \phi(s) 
 \cE^{A_1, A_2}_{ r}(s / \sqrt N)
ds \right| $$
$$\leq  \|A_1\|_{C^{k_1}} \|A_2\|_{C^{k_1}} 
\|\phi\|_{C^{k_2}} { \Vol(J)} \delta_N N^{-(d+r)/2} $$
where
$$ \cE_{ r}(s) = \cE^{A_1, A_2}_{ r}(s)= \fg(s) 
{  \sum_{p=0}^{r} \frac{P_p^{A_1, A_2} (s)}{N^{p/2}} }, $$
$\fg(\cdot)$ is a centered Gaussian density with positive definite covariance matrix
and $P_p(s)$ are polynomials in $s$ whose coefficients
are bilinear forms in $(A_1, A_2)$,
{bounded in absolute value by $C \|A_1\|_{C^{k_1}} \|A_2\|_{C^{k_1}} $,} and $P_0^{A_1,A_2}(s) = \mu(A_1) \mu(A_2)$.
\end{definition}

\begin{definition}
We say that $(f, \tau)$ satisfies a {\em mixing averaged double Edgeworth expansion of order $r$}
if there are constants $k_1, k_2$ 
and a sequence $\delta_N \to 0$ so that
for any functions $\phi_i = \phi_i(N_i) \in C^{k_2}(\mathbb R)$ 
supported on the interval $J_i =J_i(N_i)$ ($i = 1,2$),
the expression
$$ \cI_{A_1, A_2, A_3, \phi_1, \phi_2}(N_1,N_2):=
\mu(A_1(x) A_2(f^{N_1} (x)) A_3(f^{N_2} (x)) \phi_1(\tau_{N_1}(x)) \phi_2(\tau_{N_2}(x)) ) $$
satisfies
$$
\Bigg| \cI_{A_1, A_2, A_3, \phi_1, \phi_2}(N_1,N_2) $$
\begin{eqnarray*} - \iint &&
\phi_1(s_1) \fg \left( \frac{s_1}{\sqrt{N_1}} \right) 
\phi_2(s_2) \fg  \left( \frac{s_2 - s_1}{\sqrt{N_2 - N_1}} \right)  \\
&&
N_1^{-d/2} N_2^{-d/2}\sum_{p_1, p_2 = 0}^{ r}
\frac{P_{p_1,p_2}^{A_1, A_2, A_3} (s_1/\sqrt{N_1}, (s_2-s_1)/\sqrt{N_2-N_1}) }{N_1^{\frac{p_1}{2}}(N_2 - N_1)^{\frac{p_2}{2}}} ds_1 ds_2
\Bigg|
\end{eqnarray*}
\begin{eqnarray*}
&\leq& \left(\prod_{j=1}^3 \|A_j\|_{C^{k_1}} \right) \left( \prod_{i=1,2}
\|\phi_i\|_{C^{k_2}} { \Vol(J_i)} \right) \\
&&\delta_{\min \{ N_1, N_2 - N_1 \}} (\max\{N_1, N_2 - N_1 \})^{-d/2} (\min\{N_1, N_2 - N_1 \})^{-(d+r)/2} 
\end{eqnarray*}
where
 $P_{p_1,p_2}^{A_1, A_2, A_3}(s_1,s_2)$ are polynomials in $s_1, s_2$ whose coefficients
are bounded trilinear forms in $(A_1, A_2,A_3)$, 
bounded in absolute value 
by $\DS C \prod_{j=1}^3 \|A_j\|_{C^{k_1}} $,
and $$P_{0,0}^{A_1,A_2,A_3}(s) = \mu(A_1) \mu(A_2) \mu(A_3) .$$
\end{definition}

We will use the following hypotheses.
 
\begin{itemize}
\item[(A1)] $(f, \tau)$ satisfies a mixing averaged Edgeworth expansion of order $r_1$;

\item[(A1')]  $(f, \tau)$ satisfies a mixing averaged double Edgeworth expansion of order $r_1$;

\item[(A2)] For each {  $\delta>0$, we have} $\mu(|\tau_N|>N^{1/2+\delta})={\rm O}_{\delta}(N^{-r_2})$;

\item[(A3)] There are constants $\beta<1$ and $k_3\in \reals^+$ such that if $B\in C^{k_3}(Y)$ 
has zero mean, then for any $T \in \reals_+$,
$S_T^B(y):=\int_{s \in [0,T]^d} B(G_s y) ds$ satisfies
$$ \nu\left( \max_{t \in \reals, |t|<T}|S_t^B|>T^{d \beta} \right)<\frac{C \|B\|_{C^{k_3}}}{T^{r_3}}.$$

\item[(A3')] There exist constants $\beta <1$, $k_3 \in \mathbb R^+$ so that
if $B\in C^{k_3}(Y)$ 
has zero mean, then for any positive integer $M$ there is some constant $C = C_M$ so that for any $T \in \reals_+$,
$$\nu( y: |S_T^B| > T^{d\beta}) \leq C T^{-M}.
$$


\item[(A4)] $ \mu(A_1(x) A_2(f^N x))-\mu(A_1)\mu(A_2)=
{\rm O}\left(\|A_1\|_{C^{k_1}}\|A_2\|_{C^{k_1}}
{ N^{-r_4}}
\right).$
\end{itemize}

Given $H,H_1, H_2: X\times Y\to\reals$ let
\begin{equation}
\label{def:rho}
\rho_{H_1, H_2}({ N})=\zeta(H_1 (H_2\circ F^N))-\zeta(H_1) \zeta(H_2).
\end{equation}

\begin{theorem}
\label{thm:mixingspeed}
For $i=1,2,3,4$, assume
(Ai) 
with
\begin{equation}
\label{eq:rcond}
r_i > d(1 - \beta )
\end{equation}
(noting that $r_1$ is an integer).
Then there exists $K$ such that if $H_j\in C^{K}(X\times Y)$, then for any $\delta>0$
there is some $C_{\delta}$ so that
$$ \left|\rho_{H_1, H_2}(N)\right|\leq C_\delta \|H_1\|_{C^{K}} \|H_2\|_{C^{K}}
N^{d\frac{\beta-1}{2} +\delta}. $$
\end{theorem}

\begin{theorem}
\label{thm:mixingspeed2}
Assume (A1') with
\begin{equation}
\label{eq:r1'cond}
r_1 \in \mathbb N, \quad r_1 > 2d(1 - \beta )
\end{equation}
and (A2), (A3'), (A4) with $r_2,r_4$ satisfying \eqref{eq:rcond}. Then there exists $K$ such that if $H_j\in C^{K}(X\times Y)$, then for any $\delta>0$
there is some $C_{\delta}$ so that
\begin{equation}
\label{ErgMixOptimal}
 \left|\rho_{H_1, H_2}(N)\right|\leq C_\delta \|H_1\|_{C^{K}} \|H_2\|_{C^{K}}
N^{d(\beta-1) +\delta}. 
\end{equation}
\end{theorem}

The proofs of the above results use integrations by parts combined with various versions of (A1) and (A3).
The exponents and the ideas of the proofs are similar to those appearing in \cite{DLN19}, section 4.

\subsection{Proof of Theorem \ref{thm:mixingspeed}. Case of $d = 1$}
Let $\psi$ be a {  $C^{\infty}$ function such that $0\leq \psi(s)\leq 1,$ 
$\psi(0)=0$ and $\psi(1)=1$.
Given $L > 0$, let
$$\psi_L(s)=
\begin{cases}
\psi(s+L+1) & \text{ if } s \in [-L-1, -L]\\
1 & \text{ if } s \in (-L, L)\\
1 - \psi(s-L) & \text{ if } s \in [L, L+1]\\
0 & \text{ otherwise.}
\end{cases}$$}
By Corollary \ref{CrProductReduction} and (A4),
it suffices to consider the case
\begin{equation}
\label{HProduct}
H_j(x,y)=A_j(x) B_j(y)\text{ where }\nu(B_j)=0. 
\end{equation}
with $A_j, B_j \in C^{k_3}$. Without loss of generality we can assume $k_3 \geq k_2$, where $k_2$ is given by (A1).

Let $L=N^{1/2+\delta}.$ Then
$$ \rho_{H_1, H_2}(N)=
\iint A_1(x) A_2(f^N x) B_1(y) B_2(G_{\tau_N(x)}y)d\mu (x) d\nu (y)
$$
\begin{equation}
\label{truncintegral}
=
\iint A_1(x) A_2(f^N x) B_1(y) B_2(G_{\tau_N(x)}y) \psi_L(\tau_N(x)) d\mu (x)  d\nu (y)
\end{equation}
$$+
\iint A_1(x) A_2(f^N x) B_1(y) B_2(G_{\tau_N(x)}y) (1-\psi_L(\tau_N(x))) d\mu (x) d\nu (y).
$$
The integrand in the last line is zero unless $|\tau_N(x)|\geq L,$ so by (A2) the last line is
$$ {\rm O}(\|H_1\|_{C^0} \|H_2\|_{C^0} N^{-r_2}) $$
and so we need only to bound \eqref{truncintegral}.
First, observe that we can restrict the integral to $\brY$, the set of points where 
$$|S^{B_2}_t(y)|<L^{d\beta} = L^{\beta}\text{ for }t\in [-L, L]. $$
Indeed, by (A3), the integral over $Y \setminus \brY$ is in
\begin{equation}
\label{eq:a3err}
{\rm O}(\|H_1\|_{C^{0}} \|H_2\|_{C^{0}} L^{-r_3})
\end{equation}
and so is negligible.
Next observe that \eqref{truncintegral}, restricted to $ \brY$ is of the form 
$$ \int_{\brY} \cI_{A_1, A_2, \phi_y}(N) d\nu(y) \quad\text{with}\quad
\phi_y(s) =B_1(y) B_2(G_s y) { \psi}_L({s}). $$
Now by \eqref{eq:rcond}, $r_1 \geq 1$ and so
by (A1), the above expression can be replaced by
$$  
N^{-1/2}\int_{\brY} \left(\int_{-\brL}^\brL \phi_y(s) \cE_1(s/\sqrt{N}) ds \right) d\nu(y) $$
with error 
\begin{equation}
\label{eq:Edgeworthinuse}
 {\rm o}\left(\|A_1\|_{C^{k_1}}\|B_1\|_{C^{k_0}}\|A_2\|_{C^{k_1}}\|B_2\|_{C^{k_2}} {\brL}N^{-1} \right)
= {\rm o}(N^{\frac{\beta -1 }{2} + \delta})
\end{equation}
where
 $\brL=L+1$.
Integrating by parts, we obtain
$$ \int_\brY \left(\int_{-\brL}^\brL \phi_y(s) \cE_1(s/\sqrt{N}) \frac{ds}{\sqrt{N}} \right) d\nu(y) $$
$$=
-\int_\brY \left(\int_{-\brL}^\brL \cE_1'(s/\sqrt{N}) \tS_y(s) \frac{ds}{N} \right) d\nu(y) 
+{\rm O}\left(\|H_1\|_{C^0} \|H_2\|_{C^0} L \fg(L/\sqrt{N}) \right) $$
where $\tS_s(y)=B_1(y) \int_0^s \psi_L(u) B_2(G_u y) du.$
Since $$ \tS_s 1_{|s| \leq L}=B_1(y) S_s^{B_2}(y) 1_{|s| \leq L}$$ 
it follows from the definition of $\brY$
that the last integral is
$${\rm O}\left(\|A_1\|_{C^{k_1}} \|B_1\|_{C^{0}} \|A_2\|_{C^{k_1}} \|B_2\|_{C^{k_3}} \frac{L^{1+\beta}}{N}  
\right). $$
This completes the proof of the theorem. 

\subsection{Proof of Theorem \ref{thm:mixingspeed}. Case of $d \geq 2$}

We follow the approach of the one dimensional case. 
Let us assume \eqref{HProduct}
(the general case follows from Corollary \ref{CrProductReduction}).
Now $\tau \in \mathbb R^d$ and so we define
$$
\psi_L(s) = \prod_{j=1}^d \psi_L(s_j) \text{ for } s = (s_1,...,s_d)
$$
Let $\bar Y$ be defined as
$$\brY=\{y: |S_t^{B_2}(y)|<L^{d\beta} \text{ for }t\in [-L, L] \}. $$
Next we claim
$$
\rho_{H_1, H_2}(N) \approx N^{-d/2}\int_{\brY} \left(\int_{s \in [-\brL,\brL]^d} \phi_y(s) \cE_{r_1}(s/\sqrt{N}) ds \right) d\nu(y)
$$
where $a_N \approx b_N$ means $|a_n - b_N| =  o(\| H_1\|_{C^{k_1}}  \| H_2\|_{C^{k_3}}N^{d \frac{\beta -1}{2} + \eps})$. 
Indeed, repeating the argument for $d=1$, the error term
\eqref{eq:a3err} remains valid and the
error term corresponding to \eqref{eq:Edgeworthinuse} is 
${\rm O}(\brL^dN^{- (d+r_1)/2})$ which is in $o(N^{d(\beta - 1) /2 + \delta})$ by the assumption
\eqref{eq:rcond}.

Performing $d$ integrations by parts, one in each coordinate direction, we conclude
$$
\rho_{H_1, H_2}(N) \approx - N^{-d}\int_{\brY} \left(\int_{s \in [-\brL,\brL]^d} \tS_s(y) \frac{\partial^d}{\partial s_1 ... \partial s_d} \cE_{r_1}(s/\sqrt{N}) ds \right) d\nu(y).
$$
Now by the definition of $\brY$,
$$\rho_{H_1, H_2}(N) = {\rm O}(\| H_1\|_{C^{k_1}}  \| H_2\|_{C^{k_3}}  N^{-d}L^{d(1 + \beta)}),$$
and the theorem follows.

\subsection{Proof of Theorem \ref{thm:mixingspeed2}. Case of $d=1$}
\label{sec:devergdim1}

Assume \eqref{HProduct} (the general case follows from Corollary \ref{CrProductReduction}).

For fixed $y$, let us write 
$$
\sigma_{N} = \sigma_{N} (y)= \int H_1 ( F^{N}(x,y) )H_2 ( F^{2N}(x,y)) d \mu(x)
$$
so that
$$ \rho_{H_1,H_2}(N) = \zeta(H_1 (H_2\circ F^{N}))=\int \sigma_N(y) d\nu(y). $$

We will
prove that for any $\delta >0$ and for any $y \in \brY$,
\begin{equation}
\label{Variancebound}
 \sigma_{N} = {\rm o} (N^{{\beta - 1}+\delta})
\end{equation}
where $\brY$ (to be defined later) satisfies 
\begin{equation}
\label{eq:brY}
\nu (\brY) > 1 - N^{-100} 
\end{equation}
(and so the contribution of its complement is negligible).
As in the case of Theorem \ref{thm:mixingspeed}, the constant in the convergence in \eqref{Variancebound} can be bounded above by 
$$C_{\delta} \| A_1 \|_{C^{k_1}} \| A_2 \|_{C^{k_1}}
 \| B_1 \|_{C^{k_3}} \| B_2 \|_{C^{k_3}}. $$
 To simplify formulas, we do not indicate this dependence in the sequel.

Denote
$$
Y_{L,\eta} =
\{ y \in Y: \exists t \in \reals: |t| \in [L^{\eta}, L]: |S_t^{B}| > t^{\beta + \eta}\}.
$$
Next we claim that for any $\eta >0$ and for any $M$ there is some $C$ so that $\nu(Y_{L,\eta}) <C L^{-M}$.
To prove this claim, observe that for $y \in Y_{L,\eta}$ there is some $t_* = t_*(y)$ with 
$|t_*| \in [L^{\eta}, L]$ and $|S_{t_*}^{B}(y)| >  t_*^{\beta + \eta}$. 
Then $|S_{\lfloor t_* \rfloor }^{B}(y)| >  \frac12 \lfloor t_* \rfloor^{\beta + \eta}$
and so 
$$
Y_{L,\eta} \subset \bigcup_{k=\lfloor L^{\eta} \rfloor}^{\lceil L \rceil} Y_{L, \eta, k}, \text{ where }
Y_{L, \eta, k} = \left\{ y \in Y: |S_k^{B}(y)| > \frac12 k^{\beta + \eta} \text{ or } |S_{-k}^{B}(y)| > \frac12 k^{\beta + \eta}\right\}.
$$
Now we apply (A3'), with $M$ replaced by $(M+1)/\eta$ to conclude
$$\nu(Y_{L, \eta, k}) < 2C k^{-(M+1)/\eta} < C L^{-M-1}$$ 
for all $k\geq \lfloor L^{\eta} \rfloor$.
The claim follows.

Next, define
$$
\brY = Y \setminus \bigcup_{l=0,1,...,\lfloor N \rfloor} G_l^{-1} (Y_{N^{1/2+ \eps},\delta/4} )
$$
with a small $\eps = \eps(\delta)$.
By the previous claim, $\brY$ satisfies \eqref{eq:brY}.

Denote
$L_1 = N^{1/2 + \eps}$, $L_2 = 2N^{1/2 + \eps}$ and $\bar L_i = L_i+1$.
We start by computing

$$
\sigma_{N} = e_1 + $$
$$ 
+ \int A_1(f^{N}(x)) A_2(f^{ 2N}(x)) B_1(G_{\tau_{N}} (y)) B_2(G_{\tau_{2N}} (y)) 
\psi_{L_1}(\tau_{N}) \psi_{L_2}(\tau_{2N } ) d \mu(x)
$$
$$ = e_1 +
\mathcal I_{1,A_1,A_2,\phi_{y,1}, \phi_{y,2} }(N, 2N)
$$
where 
$$
\phi_{y,i}(s) = B_i(G_s(y)) \psi_{L_i}(s),
$$
and the error term $e_1$ satisfies 
\begin{equation}
\label{eq:e1}
|e_1| = {\rm O}\left(N^{-r_2}\right) = {\rm o}(N^{\beta -1})
\end{equation}
by (A2).

Now using (A1'), we derive
$$
\sigma_{N} = e_1 + e_2 +
\sum_{p_1, p_2 = 0}^{ r_1} \frac{1}{N^{\frac{p_1+p_2+{2}}{2}}} 
\mathcal J,
$$
where
$$
\mathcal J = 
\int_{-\brL_1}^{\brL_1}
\phi_{y,1}(s_1) \fg \left( \frac{s_1}{\sqrt{N}} \right) 
\int_{-\brL_2}^{\brL_2}
\phi_{y,2}(s_2) \fg \left( \frac{s_2 - s_1}{\sqrt{N}} \right) P_{p_1,p_2}^{1, A_1, A_2}  
\left( \frac{s_1}{\sqrt{N}}, \frac{s_2-s_1}{\sqrt{N}}\right) ds_2
ds_1,
$$
and where by the error term in (A1') and by \eqref{eq:r1'cond}, $e_2$ satisfies
\begin{equation}
\label{eq:e2}
|e_2| = 
{\rm O}(\brL_1 \brL_2 N^{-1/2}N^{-(1+ r_1)/2}) = 
{\rm O}\left(N^{2\eps-r_1/2}\right) = {\rm o}(N^{\beta -1 + \delta}).
\end{equation}
Next, we write the integral w.r.t. $s_2$ in $\cJ$ as
$$\cJ_1 + \cJ_2 = 
\int_{s_1-N^{1/2 + \eps}}^{s_1+N^{1/2 + \eps}}  (...) ds_2 + 
\int_{s_2 \in [-\brL_2, \brL_2] \setminus [s_1- N^{1/2 + \eps}, s_1 +N^{1/2 + \eps}]} (...) ds_2 .
$$
The integrand in $\cJ_2$ is bounded by a polynomial term times
$\fg (N^{\eps})$ and so $\cJ_2$ is negligible. Now let us write
$$
\partial_2 (P \fg) (x,y) = \frac{\partial}{\partial y} (P (x,y) \fg (y)).
$$
Then using integration by parts in $\cJ_1$ we conclude that 
\begin{equation}
\label{eq:firstsumbyparts}
\sigma_{N} \approx -
\sum_{p_1, p_2 = 0}^{ r_1} \frac{1}{N^{\frac{p_1+p_2+{ 3}}{2}}} \int_{-\brL_1}^{\brL_1}
\phi_{y,1}(s_1) \fg \left( \frac{s_1}{\sqrt{N}} \right) 
\cK_{p_1, p_2} (s_1) ds_1,
\end{equation}
where
$$ \cK(s_1) = 
\cK_{p_1,p_2} (s_1) := 
$$
$$
\int_{s_1-N^{1/2 + \eps}}^{s_1+N^{1/2 + \eps}}{S}^{ B_2}_{s_2 - s_1}(G_{s_1}y) 
\left[{ \partial_2 \left(P_{p_1,p_2}^{1, A_1, A_2} \fg \right) }
 \left( \frac{s_1}{\sqrt{N}}, \frac{s_2-s_1}{\sqrt{N}}\right) \right]  ds_2
$$
$$
=\int_{-N^{1/2 + \eps}}^{N^{1/2 + \eps}} S^{B_2}_{u}(G_{s_1}y) 
\left[{ \partial_2 \left(P_{p_1,p_2}^{1, A_1, A_2} \fg \right) }
\left( \frac{s_1}{\sqrt{N}}, \frac{u}{\sqrt{N}}\right) \right]  d u
$$
and $\approx$ means that  the difference between the two sides is in ${\rm o}(N^{\beta -1+\delta})$.

Using the fact that $y \in \brY$ and assuming that $\eps = \eps (\delta)$ is small enough, we have
\begin{equation}
\label{Kest}
\cK_{p_1,p_2} (s_1) = {\rm O}(N^{\frac{1 + \beta}{2} + \delta/2})
\end{equation}
for any $p_1, p_2$. If $p_1 + p_2 \geq 1$, then by \eqref{Kest}, the term corresponding to $p_1, p_2$ in \eqref{eq:firstsumbyparts}
is 
$$
{\rm O}(N^{-2} N^{1/2 + \eps} N^{\frac{1+ \beta}{2} + \delta/2}) = o (N^{\beta - 1 + \delta}).
$$
Next, we claim
\begin{equation}
\label{KestPrime}
\cK_{0,0}' (s_1) = {\rm O}\left( N^{\frac{\beta}{2} + \delta/2}\right). 
\end{equation}
Note that by (A1'), $P^{1,A_1,A_2}_{0,0}(x,y) = \mu(A_1) \mu(A_2)$ and so
$$
\cK_{0,0}' (s_1) = \mu(A_1) \mu(A_2)
\int_{-N^{1/2 + \eps}}^{N^{1/2 + \eps}} 
\left[ \frac{\partial}{\partial s_1} S^{B_2}_{u}(G_{s_1}y) \right] \fg' \left(  \frac{u}{\sqrt{N}}\right)
   d u
$$
$$
= \mu(A_1) \mu(A_2)
\int_{-N^{1/2 + \eps}}^{N^{1/2 + \eps}} 
 B_2 (G_{s_1+u} y) \fg' \left(  \frac{u}{\sqrt{N}}\right)
   d u
$$
$$ 
- \mu(A_1) \mu(A_2)
\int_{-N^{1/2 + \eps}}^{N^{1/2 + \eps}} 
B_2(G_{s_1} y)  \fg' \left(  \frac{u}{\sqrt{N}}\right)
   d u.
$$
The integral in the penultimate line is 
${\rm O}\left(N^{\frac{\beta}{2}+\delta /2}\right)$ since we 
can perform one more integration by parts with respect to $u.$ 
The integral in the last line is equal to
$$\sqrt N B_2(G_{s_1} y) [\fg(N^{\eps}) - \fg(-N^{\eps})],$$
which decays rapidly (i.e. faster than any polynomial)
in $N$ and so is negligible.
Thus we have verified \eqref{KestPrime}.

Now we use \eqref{KestPrime} and an integration by parts with respect to $s_1$ to conclude that the term corresponding to $p_1 = p_2 = 0$
in \eqref{eq:firstsumbyparts} is 
$$
\approx N^{-3/2}
 \int_{-\brL_1}^{\brL_1}
S^{B_1}_{s_1}(y) 
\frac{\partial}{\partial s_1}
\left( 
\fg \left( \frac{s_1}{\sqrt{N}} \right) 
\cK_{0,0} (s_1) 
\right) ds_1.
$$
Now the definition of $\brY$ together with \eqref{Kest} and \eqref{KestPrime} imply that the last expression is 
$O(N^{\beta - 1 + \delta})$ which completes the proof of \eqref{Variancebound}.

We remark that the bound \eqref{KestPrime}
can be derived in case $p_1 + p_2 \geq 1$ as well. This was not needed in case $d=1$ but will be needed in case $d \geq 2$ which we discuss next.

\subsection{Proof of Theorem \ref{thm:mixingspeed2}. Case of $d\geq 2$}
\label{sec:devergdim2}
Assume \eqref{HProduct} (the general case follows from Corollary \ref{CrProductReduction}).

We proceed as in the case of $d=1$. That is, we need to show that
\begin{equation}
\label{thm:2sigmabd}
\sigma_N = {\rm o}(N^{d(\beta -1) + \delta})
\end{equation}
for $y \in \brY$ where $\brY$ satisfies
\begin{equation}
\label{eq:100d}
\nu(\brY) > 1- N^{-100d}.
\end{equation}
First, we obtain $|e_1|= {\rm O}(N^{-r_2}) = {\rm o}(N^{d(\beta -1)})$
as in \eqref{eq:e1}.
Similarly, \eqref{eq:e2} reads as
$$|e_2| = {\rm O}(\brL_1^d \brL_2^d N^{-d/2} N^{-(d+r_1)/2}) = 
{\rm O}(N^{d \eps -r_1/2})
={\rm o}(N^{d(\beta -1) + \delta})$$
by \eqref{eq:r1'cond} and by assuming that $\eps = \eps(\delta,d)$ is small.
Next, we write
$$
\bar \partial_2 (P \fg) (x,y) = \frac{\partial^d}{\partial y_1 ... \partial y_d} (P (x,y) \fg (y)).
$$
Then as in \eqref{eq:firstsumbyparts}, we derive
\begin{equation}
\label{eq:firstsumbyparts2}
\sigma_{N} \approx -
\sum_{p_1, p_2 = 0}^{ r_1} {N^{- \frac{p_1+p_2+{ 3d}}{2}}} \mathcal J_{p_1,p_2},
\end{equation}
where $\approx$ means that the difference between the two sides is in $o(N^{d(\beta -1)+\delta})$ and 
$$\cJ_{p_1,p_2} = \int_{s_1 \in [- \brL_1, \brL_1]^d}
\phi_y(s_1) \fg \left( \frac{s_1}{\sqrt{N}} \right) 
\cK_{p_1, p_2} (s_1) ds_1,
$$
where
$$ 
\cK_{p_1,p_2} (s_1) = 
$$
$$
\int_{u \in [-N^{1/2 + \eps},N^{1/2 + \eps}]^d} 
S^{B_2}_{u}(G_{s_1}y)   
\left[{ \bar \partial_2 \left(P_{p_1,p_2}^{1, A_1, A_2} \fg \right) }
\left( \frac{s_1}{\sqrt{N}}, \frac{u}{\sqrt{N}}\right) \right]  du,
$$
and for $u \in \reals^d$,
$$S_u^B(\tilde y)
= \int_{0 \leq v_i \leq |u_i|} B(G_{v_1 sgn(u_1), ..., v_d sgn(u_d)}(\tilde y)) dv_1...dv_d
$$
where $sgn$ is the sign function ($sgn(w) = -1$ if $w<0$ and $sgn(w) = 1$ if $w >0$).
For $I = \{ i_1,...,i_{|I|}\} \subset \{ 1,2,...,d\}$, let us
write 
$$
\partial^{I} = \frac{\partial }{\partial s_{1,i_1} ... \partial s_{1,i_{|I|}}}, \quad \bar \partial =
\partial^{\{ 1,...,d\}}.
$$
We use $d$ integrations by parts with respect to the variables $s_{11}, ..., s_{1d}$ to write
\begin{equation}
\label{eq:secondsumbyparts2}
\cJ_{p_1,p_2}  =
\int_{s_1 \in [- \brL_1, \brL_1]^d} S_{s_1}^{B_1} (y) 
 \bar \partial \left[\fg \left( \frac{s_1}{\sqrt{N}} \right) 
\cK_{p_1, p_2} (s_1) \right] ds_1.
\end{equation}
We will show that for any $I \subset \{1,...,d\}$ and for any $p_1, p_2$,
\begin{equation}
\label{eq:00est}
| \partial^{I} \cK_{p_1,p_2}| \lesssim N^{\frac{d}{2}(\beta + 1) - \frac{|I|}{2}}
\end{equation}
where
$a_N \lesssim b_N$ means that $a_N < b_N N^{\delta /2}$ (assuming that $\eps = \eps(\delta)$ is small
enough). 
Assume first that \eqref{eq:00est} hold. Then observe that
$$
\left| \bar \partial \left[\fg \left( \frac{s_1}{\sqrt{N}} \right) 
\cK_{p_1, p_2} (s_1) \right] \right|
\lesssim N^{\frac{d \beta}{2}}.
$$
Substituting this estimate to \eqref{eq:secondsumbyparts2}, we obtain
$$
\left| \cJ_{p_1,p_2} \right| \lesssim N^{d/2} N^{\frac{d \beta}{2}} 
N^{\frac{d \beta}{2} },
$$
which, implies \eqref{thm:2sigmabd}.
Thus it remains to prove
\eqref{eq:00est}.

Assume that $\fg$ is the standard Gaussian density (if this is not the case, we can 
compute all integrals on a parallelepiped of side length $cN^{1/2 + \eps}$, then apply a linear change of variables
to reduce to the case of standard Gaussian).
To prove \eqref{eq:00est} we write 
$$ h = \bar \partial_2 \left(P_{p_1,p_2}^{1, A_1, A_2} \fg \right).$$
Recall that $I = \{ i_1, ..., 1_{|I|} \}$, the set of indices $i$ such that we are differentiating with respect to $s_{1,i}$, is given.
We need to differentiate the integrand in $\cK$, which is a product. 
Let $I' = \{ i'_1, ..., i'_{|I'|} \}  \subset I$
denote the set of indices $i'$ so that we differentiate
the term $S_u^{B_2}(G_{s_1}(y))$ with respect to $s_{1,i'}$. For $i \in I \setminus I'$, we differentiate 
$h$ with respect to $s_{1,i}$.
We also write $J = \{1,...,d \} \setminus I$
and $J' = \{1,...,d \} \setminus I'$.
Performing the differentiation, we find

 \begin{align}
&\partial^{I} \cK_{p_1,p_2}  = \nonumber \\
& \sum_{I': I' \subset I} \int_{u \in [-N^{1/2 + \eps},N^{1/2 + \eps}]^d}
\int_{w_{j'} \in [0, |u_{j'}|] \text{ for }j' \in J'}
\sum_{\delta_{i'} \in \{0, 1\} \text{ for }i' \in I'} (-1)^{|I'| - \sum \delta_{i'}} \nonumber \\
& B_2(G_{(i': s_{1i'} + \delta_{i'} u_{i'};j': s_{1j'} + w_{j'} sgn(u_{j'}))} (y) ) 
\left[ \partial^{I \setminus I'} h  \left(\frac{s_1}{\sqrt N}, \frac{u}{\sqrt N} \right)  \right] 
dw_{j'} du, \label{eq:prodrule}
\end{align}
where in the subscript of $G$ the notation $(i': a_{i'};j': b_{j'})$ means that for coordinates $i' \in I'$ we use $a_{i'}$
and for $j' \in J'$, we use $b_{j'}$. 
Note that
\begin{equation}
\label{eqhtilde}
\partial^{I \setminus I'} h  \left(\frac{s_1}{\sqrt N}, \frac{u}{\sqrt N} \right)
=  N^{- \frac{|I|- |I'|}{2}} \tilde h  \left(\frac{s_1}{\sqrt N}, \frac{u}{\sqrt N} \right),
\end{equation}
where 
$$
  \tilde h(x,y) =
\frac{\partial^{|I| -|I'|}}{\partial x_{i_1'} ... \partial x_{i'_{|I'|}}}
 \frac{\partial^d}{\partial y_1 ... \partial y_d} (P (x,y) \fg (y)).
$$

Now assume there is some $i'$ so that $\delta_{i'} = 0$. Then $B_2(...)$ does not depend on $u_{i'}$
and so performing the integral with respect to $u_{i'}$ first, we obtain
\begin{equation}
\label{eq:delta0}
 \int_{u_i \in [-N^{1/2 + \eps},N^{1/2 + \eps}]} 
 \tilde h  \left(\frac{s_1}{\sqrt N}, \frac{u}{\sqrt N} \right)
 du_i 
\end{equation}
$$
=\sqrt N \sum_{a = 1,2} (-1)^{a} \tilde h_i 
\left( \frac{s_1}{\sqrt N}, \left( \frac{u_1}{ \sqrt N}, ..., \frac{u_{i-1}}{ \sqrt N},  (-1)^{a} N^{\varepsilon}, 
\frac{u_{i+1}}{ \sqrt N}, ..., \frac{u_{d}}{ \sqrt N} \right)  \right),
$$
where
$$
\tilde h_i (x,y) = 
\frac{\partial^{|I| -|I'|}}{\partial x_{i_1'} ... \partial x_{i'_{|I'|}}}
 \frac{\partial^{d-1}}{\partial y_1 ... \partial y_{i-1} \partial y_{i+1} ... \partial y_d} (P (x,y) \fg (y)).
$$
Recalling that $\fg (y) =  \frac{1}{(2 \pi)^{d/2}} \exp \left(- \sum_{i=1}^d y_i^2/2 \right) $,
we see that $\tilde h_i(x,y)$ decays rapidly as $y_i \to \infty$ (i.e. faster than any polynomial).
Since we have $|y_i| = N^{\eps}$, \eqref{eq:delta0} decays rapidly as $N \to \infty$. Thus this term, even
when integrated with respect to all other variables, decays rapidly and consequently we can neglect all terms 
in \eqref{eq:prodrule} where there is some $i'$ so that $\delta_{i'} = 0$.

It remains to study the case
when $\delta_{i'} = 1$ for all $i' \in I'$.  Then we perform the integrals in \eqref{eq:prodrule} with respect to
$w_{j'}, j' \in J' $ and we integrate by parts with respect to $u_{i'}, i' \in I'$ to obtain that
$$
| \partial^{I} \cK_{0,0}  
- 
\mathcal I |
$$
decays rapidly as $N \to \infty$, where
$$
\mathcal I = \int_{u_{j'}, j' \in J'}  
\int _{u_{i'}, i' \in I' } 
S^{B_2}_{b} (y)
\left[ \partial^{I} h  \left(\frac{s_1}{\sqrt N}, \frac{u}{\sqrt N} \right)  \right]
du_{i'} du_{j'} 
$$
and 
$$
b = (i': N^{1/2+ \eps}, j': u_{j'}).
$$
As in \eqref{eqhtilde},
we have
\begin{equation}
\label{defhhat}
\partial^{I } h  \left(\frac{s_1}{\sqrt N}, \frac{u}{\sqrt N} \right)
=  N^{- \frac{|I|}{2}} \hat h  \left(\frac{s_1}{\sqrt N}, \frac{u}{\sqrt N} \right)
\end{equation}
where 
$$
  \hat h(x,y) =
\frac{\partial^{|I|}}{\partial x_{i_1} ... \partial x_{i_{|I|}}}
 \frac{\partial^d}{\partial y_1 ... \partial y_d} (P (x,y) \fg (y)).
$$
Note that we can assume $|S^{B_2}_{b }| \lesssim N^{d \beta /2}$.
Indeed, we can subdivide the rectangular box with opposite corners $0$ and $b$ into small cubes of side length $N^{\eps}$ 
and we can assume that the integral of $G_s(y)$ over all of the boxes is smaller than $N^{d\eps\beta}$ for $y \in \brY$
by (A3') ($\brY$ satisfies \eqref{eq:100d} similarly to the case $d=1$). 
Combining this observation with \eqref{defhhat}, we conclude
$$
|\mathcal I| \leq N^{\frac{d \beta - |I|}{2 }}
\int_{ u\in [-N^{1/2 + \eps}, N^{1/2 + \eps}]} \| \hat h\|_{\infty} du
\leq C  N^{\frac{d (\beta +1) - |I|}{2 } + \delta/2}
$$
if $\varepsilon (\delta)$ is small enough.
This completes the proof of
\eqref{eq:00est} and so the theorem follows.


\section{Toral translations and related systems}
\label{ScToral}

\subsection{Rapid mixing}
\label{SSToral}
Let $f$ be an Axiom A 
diffeomorphism, and $\mu$ be a Gibbs measure with H\"older potential.
Let $Y=\Tor^m$ and $G_t$ be a $d$-parameter flow:
$\DS G_{(t_1, \dots, t_d)}(y)=y+\sum_{j=1}^d \alpha_j t_j $
for some $\alpha_1,\dots ,\alpha_d\in \reals^m.$ 
Note that $G_t$ has discrete spectrum, so it is far from being mixing.
However, according to \cite{Dol02} the mixing properties of the corresponding skew products are typically 
much better than the results obtained in Section \ref{ScMix-MF} for the case of the mixing fibers.
Namely, let $\Pi$ be the linear subspace generated by 
$\alpha_1, \dots, \alpha_d.$ We say that $\Pi$ is {\em Diophantine} if there exist numbers $K, s$ such that 
for any unit vector $v\in \Pi$ for any $k\in \integers^m$ we have
$\DS  |\langle v, k \rangle|\geq K |k|^{-s}. $

\begin{proposition}
\label{PrToral}
(\cite{Dol02})
If $\Pi$ is Diophantine, then $F$ is rapidly mixing except for the set $\tau: X\to \Pi$ lying in an infinite 
codimension submanifold.
\end{proposition}

Next, we describe an application of this result.

\subsection{Constant suspensions in the fiber}
Again we take $f$ as in \S \ref{SSToral} but now we consider constant suspensions acting in the fiber.
That is let $\cG^{\bn}$ be a $\integers^d$ an exponentially mixing action on a manifold $\cY$ preserving a measure $\tnu$, let 
$Y=\cY\times \reals^d/\sim$ where $\sim$ is the identification $(\ty, z+\bn)\sim(\cG^{\bn} \ty, z).$
Let $G^t$ be the action $(\ty, z)\to (\ty, z+t)$. It preserves measure $d\nu=d\tnu\; dz.$ 



Given a $T, T^{-1}$ map as above, consider an associated action $\cF$ on $X\times \Tor^d$
given by $\cF(x, \theta)=(f x, \theta+\tau(x)).$

\begin{proposition}
\label{PrMixSusp}
Suppose that $\cF$ is rapidly mixing. Then 
\eqref{TTInvGreenKubo} holds.
\end{proposition}

\begin{proof}
Split $H=\brH+\tH$ where
$\DS \brH(x, z)=\int H(x, \ty, z)d\tnu(\ty). $
Note that $G_t$ and hence $F$ preserves this splitting and that $\brH$ is $\integers^d$ invariant,
because $\cG^{\bn}$ preserves $\tnu$ and
$$ \int H(x, \ty, z+\bn)d\tnu(\ty)=\int H(x, \cG^{\bn} \ty, z)d\tnu(\ty)=\brH(x, z).$$
It follows that 
$$\rho_{H_1, H_2}(n)=\rho_{\brH_1, \brH_2}(n)+\rho_{\tH_1, \tH_2}(n). $$
The first term decays faster than any polynomial, because $\cF$ is rapidly mixing and
the second term is $O\left(n^{-d/2}\right)$ due to Remark \ref{RmGKRelative}.  However to apply 
the remark, we need to check that $G_t$ is exponentially mixing on the space $\BAN$ of $C^L$ 
functions such that 
$$ \int H(x, (\ty, z)) d\tnu(y)=0 \text{ for all } (x, z). $$
To check mixing, we write $t=\bn+\hat{t}$, where $\bn\in \integers^d$ and $\hat{t}$ belongs to 
the unit cube. Then
$$ \int H_1(x_1, (\ty, z)) H_2(x_2, G_t(\ty, z)) d\nu=
\iint H(x_1, (\ty, z-\hat{t})) H_2(x_2, (\cG^n \ty, z)) d\tnu(\ty) dz .$$
Integrating first with respect to $\ty$, we see that the RHS decays exponentially as needed.
\end{proof}

\section{Deviations of ergodic averages}
\label{ScDeviations}

\subsection{Mixing and deviations}
Here we recall some results about the relations of mixing and deviations of ergodic averages.


\begin{lemma}
\label{lemma:stat}
Let $X_1,X_2,...$ be a stationary sequence of random variables on a probability space 
$(\Omega, P)$ and $S_N = \sum_{k=1}^N X_k$. Assume that there are constants $C$ and 
$\rho$ such that for 
every $n$
\begin{equation}
\label{VarSt}
E(S_n^2) < C n^{2 \rho}.
\end{equation}
 Then
$S_n / n^{\max\{\rho, \frac{1}{2}\} + \eps}$ converges to zero almost surely for all $\eps >0$.
\end{lemma}

\begin{proof}
Let us assume $\rho > 1/2$ (the case $\rho\leq  1/2$ is a simple consequence).
For a positive integer $m$, let $D_m$ denote the collection of intervals of the form $I_{i,j} = [j2^i +1, (j+1)2^{i}]$ for all
non-negative integers $i,j$ so that $(j+1)2^i \leq 2^m$. By the stationarity assumption,
$$
E \left(\sum_{I \in D_m} \left(\sum_{k \in I} X_n\right)^2\right) \leq \sum_{i=0}^m 2^{m-i} E(S_{2^i}^2)
\leq C \sum_{i=0}^m 2^{m-i} 2^{2i \rho} \leq \tC 2^{2m \rho} 
$$
Now for given positive integer $n$, let $m$ be so that $2^{m-1} < n \leq 2^m$. Then the interval $[1,n]$
can be written as a disjoint union of at most $2m$ intervals from the family $D_m$. Let us denote this
collection of intervals by $D(n)$. Then by the Cauchy Schwartz inequality,
$$
S_n^2 = \left(\sum_{I \in D(n)} \sum_{k \in I} X_k\right)^2 \leq 2m \sum_{I \in D(n)} \left(\sum_{k \in I} X_k\right)^2
\leq 2m \sum_{I \in D_m} \left(\sum_{k \in I} X_k\right)^2
$$
Thus we have
$$
P(\exists n = 2^{m-1}+1,...,2^m: S_n^2 > \eta n^{2\rho + \eps})
\leq 
P(2m \sum_{I \in D_m} (\sum_{k \in I} X_k)^2 > \eta 2^{(m-1)(2\rho + \eps)}) $$
$$\leq 2m \eta^{-1} 2^{-(m-1)(2\rho + \eps)} E(\sum_{I \in D_m} (\sum_{k \in I} X_k)^2) 
\leq \tC\eta^{-1} m2^{-m \eps}
$$
Using the Borel-Cantelli lemma and the fact that $\eta >0$ is arbitrary, Lemma \ref{lemma:stat} follows. 
\end{proof}

\begin{lemma}
\label{lem:vargrowth}
Under the assumptions of Lemma \ref{lemma:stat}
suppose that
$|E(X_i X_j)|\leq C |i-j|^{-\beta}$ then \eqref{VarSt} is satisfied with 
$$\rho=\begin{cases} \frac{1}{2}, & \text{if } \beta>1, \\
1-\frac{\beta}{2} &  \text{if } \beta<1. \end{cases}. $$
\end{lemma}

\begin{proof}
\eqref{VarSt} follows since
$\DS
E(S_N^2) =N E(X_0^2)+2 \sum_{n=0}^{N-1} (N-n) E(X_0 X_n). 
$
\end{proof}

\subsection{Examples and open questions}
\label{ScExamples}
Here we describe several classes of systems satisfying our assumptions on the
base and the fiber dynamics made in previous sections.
We also present several open questions pertaining to establishing those properties in several new cases.

Mixing of the base system is required in all our results. In addition the results of Sections \ref{ScMix-MF} 
require
mixing in the fiber, so we begin with reviewing known results for mixing.

Exponential mixing is known in the following cases:
uniformly hyperbolic diffeomorphisms with Gibbs measures (\cite{Bow75, PP90});
nonuniformly hyperbolic systems admitting Young towers with exponential tails (\cite{You98});
partially hyperbolic translations on homogeneous spaces (\cite{KM99, BEG18});
contact Anosov flows \cite{Liv04} as well as Anosov flows with suitable assumptions on
Lyapunov spectrum \cite{ABV16, Ts18}; some singular hyperbolic flows \cite{AM16}; 
ergodic automorphisms of tori \cite{Kat71} and of nilmanifolds (\cite{GS14}). In all the examples of $\reals$ or $\integers$ actions
listed above, we also have multiple exponential mixing (see e.g. \cite{Dol04}) while in higher rank the multiple exponential mixing is only known for
partially hyperbolic translations on homogeneous spaces (\cite{BEG18}), (partial results for some $\integers^d$ actions are obtained in \cite{GS15}). 

Rapid mixing is known for generic Axiom A flows with Gibbs measures (\cite{Dol98-EM, Dol98-RM, FMT07}), hyperbolic flows having Young towers with exponential tails (see \cite{Mel18} and references wherein),
some singular hyperbolic flows \cite{AM19}, and generic compact group extensions of uniformly hyperbolic systems (\cite{Dol02}).

Polynomial mixing is known for nonuniformly hyperbolic diffeomorphisms and flows having Young towers with polynomial tails  (\cite{Sar02, Gou04, BBM19}),
unipotent actions (\cite{KM99, BEG18}, time changes of nilflows (\cite{FK17}), and some flows on surfaces with degenerate singularities~(\cite{FFK}). 

Additional assumptions imposed on base dynamics in various results include 
large deviations, anticoncentration, LLT and Edgeworth expansions.

An easiest way to get large deviation is to have unique ergodicity since in that case the
set in LHS  of \eqref{EqLD-Exp} is empty. A relative version of unique ergodicity is so called 
{\em uunique ergodicity} (see \cite{Dol04} for a definition),
 which holds for partially hyperbolic systems with unique measure absolutely 
continuous with respect to the unstable foliation. In this case \eqref{EqLD-Exp} holds due to \cite{Dol04}.  
Exponential large deviations also hold for non-uniformly hyperbolic systems admitting Young towers
with exponential tails for return times \cite{MN08, RBY08}, while in case the tail is polynomial, 
polynomial large deviations hold \cite{Mel09, GM14} (see also \cite{DN-Ren} where the large deviations
are discussed under a quasiindependence assumption).

Anticoncentration inequality is established 
for systems admitting Young towers provided that the return time tail has second moment
\cite{Pene09}.

The LLT is known for Axiom A diffeomorphisms with Gibbs measures (\cite{PP90}), the systems 
admitting Young tower under the assumptions that the tails admit the second moment
(\cite{SV04}) as well as flows which can be represented as suspensions of  flows
admitting nice symbolic dynamics \cite{DN-Flows} including Axiom A 
flows and certain Lorenz type attractors. 
The results of \cite{DN-Flows} can be applied to continuous time 
$T, T^{-1}$ systems  given by \eqref{TTInvFlows}.

Mixing averaged Edgeworth expansions are obtained in \cite{FP20} for systems admitting Young towers
with exponential tails. It seems that the methods of \cite{FP20} as well as \cite{DNP} could be used to
obtain the multiple expansions as well but this remains an open problem.

For fiber dynamics we require control on ergodic averages. For mixing systems such control 
can be obtain using moment estimates (cf. Lemma \ref{lemma:stat}).

Systems satisfying assumption (A3) (or (A3')) for $d=1$ include exponentially mixing systems
described above, as well as toral translations (see e.g. \cite{DFSurv}), 
 products of the last two examples \cite{CC13},
horocycle flows \cite{FF03}, translation flows 
(those flows are not smooth, however, the results of Section \ref{ScMix-EF} apply provided 
that we consider the observables which vanish near the singularities), 
 typical area preserving flows 
on surfaces (with non-degenerate singularities) \cite{For02} 
and nilflows (\cite{FlaFo}, \cite{For16}). Higher dimensional examples include 
Cartan and unipotent actions on homogeneous spaces of semisimple Lie groups (\cite{BEG18})
and multidimensional niltranslations \cite{CF15}. 

The results of this paper motivate the study of the statistical properties discussed above
for a wider class of dynamical systems. In particular, it is of interest to 

(a) construct example of systems satisfying mixing multiple Edgeworth expansion;

(b) prove mixing LLTs for partially hyperbolic systems;

(c) investigate mixing LLTs and anticoncentration bounds for parabolic systems.

\subsection{Deviations of ergodic averages for generalized $T, T^{-1}$ transformations}
 Here we illustrate the information the results obtained in this paper provide about the growth of ergodic 
sums in several special cases. In the examples below we assume that the base dynamics $f$ is given by an
Anosov diffeomorphism equipped with a Gibbs measure and for each fiber flow 
(1--10) we
give an exponent $\alpha$ such that
with probability one the ergodic sums 
of the corresponding generalized $T,T^{-1}$ transformation
grow slower than $N^{\alpha+\eps}$
for every $\eps >0$. 
This is going to be a simple
consequence of Lemmas \ref{lemma:stat} and \ref{lem:vargrowth}.
For each example we list the result that implies the assumption of 
Lemma \ref{lem:vargrowth} with a suitable $\beta$.
In case we use the results of Section \ref{ScMix-EF}, we also assume that $(f, \tau)$ 
satisfies the mixing double averaged
Edgeworth expansion of any order. Currently no examples of such systems is known but we expect
this property to hold for large class of map (cf. e.g. the computations in \cite{DNP}).

\begin{enumerate}
\item Anosov diffeomorphisms. In this case we have exponential mixing (\cite{Bow75, PP90});

(a) zero drift \thus $\alpha=3/4$ (Thm \ref{ThLLT-MF});
(b) positive drift \thus  $\alpha=\frac{1}{2}$  (Thm \ref{ThDrift}).
\item Diophantine toral translations--here $(A3')$ holds for any $\beta>0$ and so 
$\alpha=1/2$ by Thm \ref{thm:mixingspeed2}  (cf. also Prop \ref{PrToral}).
\item Product of Anosov diffeomorphisms and toral translation\thus
$\alpha=3/4$ (Thm \ref{thm:mixingspeed2}).
\item horocycle flows (see \cite{FF03})\thus Thm \ref{thm:mixingspeed2} gives
\begin{enumerate}
\item no small eigenvalues of $\Delta$, zero drift--$(A3)$ holds for any $\beta>1/2$, 
so $\alpha=\rho_1(\beta)=3/4$;
\item smallest eigenvalue of $\Delta$ is $\lambda\in \left(0, \frac{1}{4}\right)$--$(A3)$ 
holds for any $\beta>\frac{1+\sqrt{1-4\lambda}}{2}$, so 
$\alpha=\rho_1(\beta)=\frac{1+\sqrt{1-4\lambda}}{2}$.
\end{enumerate}
\item translations flows--$(A3')$ holds for any $\beta>\lambda_2$ (\cite{For02}) 
where $\lambda_2$ is the second exponent 
of Kontsevich-Zorich cocycle. So $\alpha=\rho_1(\beta)=\frac{\lambda_2+1}{2}$
(Thm \ref{thm:mixingspeed2}).
\item partially hyperbolic translations on homogenous spaces.
In this case we have exponential mixing (\cite{KM99, BEG18});

(a) zero drift\thus $\alpha=3/4$ (Thm \ref{ThLLT-MF});
(b) positive drift\thus  $\alpha=\frac{1}{2}$  (Thm \ref{ThDrift}).
\item multidimensional Cartan actions on homogenous spaces\thus
$\frac{1}{2}$
(Thms \ref{ThLLT-MF} and \ref{ThDrift}).
\item constant suspensions of Cartan actions on tori\thus $\frac{1}{2}$ (Pr \ref{PrMixSusp}). 
\item continuous time $T, T^{-1}$ system given by \eqref{TTInvFlows}
with both base flow $\phi^t$ and fiber flow $G_t$ given by geodesic flow on a unit tangent bundle
over a negatively curve manifold:
{ $\alpha = \frac{7}{8}$ 
by Example \ref{TwoFlowEx}(b) with $k=2$. In fact, Example \ref{TwoFlowEx}(b)} shows that
for all positive integers $k$, 
we can obtain a system 
with $\alpha = 1 - 2^{-k-1}$.
\item generic higher rank 
actions on Heisenberg nilmanifolds:
$\frac{1}{2}$  
(\cite{CF15} and Thm~\ref{thm:mixingspeed2}).
\end{enumerate}

\appendix
\section{Anticoncentration large deviation bounds for subshifts of finite type}
We follow the argument in \cite{DN-Ren}.

Let $(\Sigma, \sigma)$ be a subshift of finite type, 
$\mu$ be a Gibbs measure and $\tau:\Sigma\to\reals^d$ be a H\"older function of zero mean.
We assume that for each $\ba\in\reals^d\backslash\{\bf0\}$ the function $\langle \ba, \tau\rangle$ is not a coboundary.

\begin{lemma}
\label{LmTaylorZero}
(\cite{PP90}) There are constants $c_1, \delta_0$ such that for $|\xi|<\delta_0$
\begin{equation}
\label{LaplaceTau}
\mu\left(e^{\langle\xi, \tau_N\rangle}\right)\leq e^{c_1 N\xi^2}. 
\end{equation}
\begin{equation}
\label{FourierTau}
{
|\Phi_N(\xi)|
\leq e^{-c_1 N\xi^2}, \quad \text{where} \quad
\Phi_N(\xi) = \mu\left(e^{i\langle\xi, \tau_N\rangle}\right). }
\end{equation}
\end{lemma}

\begin{corollary}
There are constants $C_2, c_2$ such that 
\begin{equation}
\label{LDTau}
 \mu(|\tau_N|>L)\leq C_2 e^{-c_2 L^2/N} 
\end{equation} 
and for each unit cube $\cQ$
\begin{equation}
\label{ACTau}
\mu(\tau_N\in \cQ)\leq \frac{C_2}{N^{d/2}}.
\end{equation} 
\end{corollary}

\begin{proof}
To prove the first inequality we may assume without the loss of generality that $d=1$ and that 
$\sqrt{N}\leq L\leq 2 c_1 \delta_0 N$ (we obtain the general result 
by increasing $C_2$ and decreasing $c_2.$) We estimate
$\mu(\tau_N>L)$, the bound for $\mu(\tau_N<-L)$ being similar.
We have that for each $\xi\in (0, \delta_0)$ 
$$ \mu(\tau_N>L)=\mu\left(e^{\xi\tau_N}>e^{\xi L}\right)\leq 
e^{-\xi L} \mu\left(e^{\xi \tau_N}\right)
\leq e^{-\xi L+c_1 N \xi^2}$$
Taking $\xi=\frac{L}{2c_1 N}$ we obtain the result.

It is enough to prove \eqref{ACTau}
for cubes of any fixed size $\rho$ since the unit cube can be covered by a finite number
of cubes of size $\rho.$ Let
$$ g(x)=\prod_{l=1}^d \left(\frac{1-\cos(\hdelta x_{(l)})}{\hdelta^2 x_{(l)}^2}\right)$$
where $\hdelta=\delta_0/d$ and $\delta_0$ is the constant from Lemma \ref{LmTaylorZero}.
Then 
$$\hg(\xi)=(\pi \hdelta)^d \prod_{l=1}^d \left(\left(1-\frac{|\xi|}{\hdelta}\right)1_{|\xi|\leq \hdelta} \right). $$
Hence for each $a$
$$\EXP(g(\tau_N-a))
=\int_{\reals^d} \hg(-\xi) e^{i\xi a} \Phi_N(\xi) d\xi \leq
\int_{|s|<\delta_0} \hg(s) |\Phi_N(s)| ds $$
since $\hg$ is real, positive, and supported inside the ball of radius $\delta_0.$  
Thus \eqref{FourierTau} implies that there is a constant $\hD$ such that
$$ \EXP(g(\tau_N-a))\leq \frac{\hD}{N^{d/2}}$$
On the other hand  $g(0)=\frac{1}{2^d}$ so there is a constant $\rho$ such that $g(x)>\frac{1}{4^d}$ on the
cube of size $\rho$ centered at $0.$ Hence if $\cQ$ is a cube of size $\rho$ centered at $a$ then
$$ \EXP(g(\tau_N-a))\geq \frac{\Prob(S_N\in \cQ)}{4^d}. $$
Combining the last two displays we obtain the result.
\end{proof}

We now prove the anticoncentration large deviation estimate with $\Theta(r)=e^{-c_4 r^2}. $

\begin{lemma}
\label{LmACLD2}
If $\cQ$ is a unit cube  centered at $z$, then
$$ \mu(\tau_N\in \cQ)\leq \frac{C_3}{N^{d/2}} e^{-c_3 z^2/N}. $$
\end{lemma}

\begin{proof} There is a constant $R$ such that
$$ \mu(\tau_N\in \cQ)\leq \mu\left( \tau_N\in \cQ, \;
|\tau_{N/2}|>\frac{|z|}{2}-R\right)+
\mu\left( \tau_N\in \cQ, \; |\tau_N-\tau_{N/2}|>\frac{|z|}{2}-R\right). $$
We will estimate the first term, the estimate of the second 
is obtained by replacing $\sigma$ by $\sigma^{-1}.$
We have
$\DS \mu\left( \tau_N\in \cQ,\; |\tau_{N/2}|>\frac{|z|}{2}-R\right)\leq
\sum_{\cC', \cC''} \mu(\cC' \cC'') ,$
where the sum is over all pairs of cylinders $(\cC', \cC'')$ such that 

(i) length$(\cC')=$ length$(\cC'')=N/2$,

(ii) there exists $\omega'\in \cC'$ such that $|\tau_{N/2}(\omega')|>\frac{|z|}{2}-R,$

(iii) there exists $\omega''\in \cC''$ such that 
$\left|\tau_{N/2}(\omega')+\tau_{N/2}(\omega'')-z\right|<2R.$

By Gibbs property
$\DS \sum_{\cC', \cC''} \mu(\cC' \cC'') \leq K \sum_{\cC', \cC''} \mu(\cC') \mu(\cC'') . $

By \eqref{ACTau} for each $\cC'$ the sum of $\mu(\cC'')$ over the cylinders $\cC''$ satisfying (iii) is smaller than
$\DS \frac{(2R)^{d} C_2}{N^{d/2}}. $ Summing over $\cC'$ satisfying (ii) and using \eqref{LDTau}, we
obtain the result.
\end{proof}

\begin{lemma}
\label{LmACLDs}
Let $\cQ_1,\dots \cQ_s$ be unit cubes centered at $z_1,\dots z_s.$ Then 
with the notation $z_0 = 0 \in \reals^d$, $n_0 = 0$,
$$ \mu\left(\tau_{n_j}\in \cQ_j \text{ for } j=1, \dots s\right)\leq
\prod_{j=1}^s \left[ \left( \frac{C_4}{(n_j-n_{j-1})^{d/2}} \right) e^{-c_4 \frac{|z_j-z_{j-1}|^2}{n_j-n_{j-1}}}\right].
$$
\end{lemma}

\begin{proof}
The LHS can be bounded by 
$\DS  \sum(\mu(\cC_1\cC_2\dots \cC_s)) $
where the sum is over all tuples of cylinders such that

(i) length$(\cC_j)=n_j-n_{j-1}$ and

(ii) On $\cC_j$, $\tau_{n_j-n_{j-1}}$ is contained in a cube of size $R$ centered at
$z_j-z_{j-1}.$ 

Using Gibbs property the last can be bounded by
$\DS K \prod_{j=1}^s \left[ \sum_{\cC_j: (i)\text{ and } (ii) \text{ hold}} \mu(\cC_j)\right] . $
Now the result follows by Lemma \ref{LmACLD2}.
\end{proof}


\begin{thebibliography}{99}
\bibitem{ABV16}  Araujo V., Butterley O., Varandas P. {\em Open sets of axiom A flows with exponentially mixing attractors,}
  Proc. AMS {\bf 144} (2016) 2971--2984.

\bibitem{AM16} Araujo V., Melbourne I.
{\em Exponential decay of correlations for nonuniformly hyperbolic flows with a $C^{1+\alpha}$ stable foliation, including the classical Lorenz attractor,} 
Ann. Henri Poincare {\bf 17} (2016) 2975--3004.

\bibitem{AM19} Araujo V., Melbourne I.
  {\em Mixing properties and statistical limit theorems for singular hyperbolic flows without a smooth stable foliation,}
  Adv. Math. {\bf 349} (2019) 212--245.

\bibitem{Arnol'd}  Arnold V. I.
{\em Topological and ergodic properties of closed 1-forms with
incommensurable periods}, Funktsional. Anal. i Prilozhen. {\bf 25} (1991), 1--12, 96.

\bibitem{BBM19} B\'alint P., Butterley O., Melbourne I.
{\em Polynomial decay of correlations for flows, including Lorentz gas examples,}
Comm. Math. Phys. {\bf 368} (2019) 55--111. 

\bibitem{BEG18} Bj\"orklund M., Einsiedler M., Gorodnik A. {\em Quantitative multiple mixing,} 
JEMS {\bf 22} (2020) 1475--1529

\bibitem{BG19} Bj\"orklund M., Gorodnik A. 
{\em Central Limit Theorems for group actions which are exponentially mixing of all orders,} 
to appear in Journal d'Analyse Mathematiques.

\bibitem{Bolt89} Bolthausen E. {\em A central limit theorem for two-dimensional random walks in random sceneries,}  Ann. Probab. {\bf 17} (1989) 108--115.

\bibitem{BDV05} Bonatti C., Diaz L. J., Viana M. {\em Dynamics beyond uniform hyperbolicity. A global geometric and probabilistic perspective,} 
Encycl. Math. Sci. {\bf 102}  (2005) xviii+384 pp. 
  
\bibitem{Bow75} Bowen R. {\em Equilibrium states and ergodic theory of Anosov diffeomorphisms,} Lect. Notes in Math. {\bf 470} (1975) Springer New York. 

\bibitem{Br05}
  Breuillard E. {\em Distributions diophantiennes et theoreme limite local sur $\reals^d$,}
    Probab. Theory Related Fields {\bf 132} (2005) 39--73.   
    
\bibitem{Bufetov} Bufetov A. {\em Limit Theorems for Translation Flows}, Annals  Math. {\bf 179} 
(2014) 431--499.
   
\bibitem{Ch95} Chernov, N. I.
{\em Limit theorems and Markov approximations for chaotic dynamical systems,} 
Probab. Theory Related Fields {\bf 101} (1995) 321--362. 

\bibitem{CM06} Chernov N., Markarian R. {\em Chaotic billiards,} 
Math. Surv. \& Monographs {\bf 127}  (2006) AMS,
Providence, RI, xii+316 pp.

\bibitem{CC13} Cohen G., Conze J.--P.
{\em The CLT for rotated ergodic sums and related processes,} 
DCDS-B {\bf 33} (2013)  3981--4002. 

\bibitem{CC17} Cohen G., Conze J.--P.
{\em CLT for random walks of commuting endomorphisms on compact abelian groups,}
J. Theoret. Probab. {\bf 30} (2017) 143--195. 


\bibitem{CF15}  Cosentino S., Flaminio L. {\em Equidistribution for higher-rank abelian actions on Heisenberg nilmanifolds,} 
J. Mod. Dyn. {\bf 9} (2015) 305--353.

\bibitem{dHKSS} den Hollander F., Keane M. S., Serafin J., Steif J. E. {\em Weak Bernoullicity of random walk in random scenery,}
  Japan. J. Math. {\bf 29} (2003) 389--406.

\bibitem{dHS97} den Hollander F., Steif J. E. 
{\em Mixing properties of the generalized $T, T^{-1}$-process,} 
J. Anal. Math. {\bf 72} (1997) 165--202.

\bibitem{dHS06} den Hollander F., Steif J. E. 
{\em Random walk in random scenery: a survey of some recent results,} 
IMS Lecture Notes Monogr. {\bf 48} (2006) 53--65.

\bibitem{Dol98-EM} Dolgopyat D. {\em On decay of correlations in Anosov flows,} Ann. Math. {\bf 147} (1998) 357--390.  
  
  
\bibitem{Dol98-RM} Dolgopyat D. {\em Prevalence of rapid mixing in hyperbolic flows,} Erg. Th. Dynam. Syst. {\bf 18} (1998) 1097--1114.

\bibitem{Dol02} Dolgopyat D. {\em On mixing properties of compact group extensions of hyperbolic systems,} Israel J. Math. {\bf 130} (2002) 157--205.  

\bibitem{Dol04} Dolgopyat D. {\em Limit theorems for partially hyperbolic systems,}
  Trans. AMS {\bf 356} (2004) 1637--1689. 

\bibitem{DDKN-Flex} 
Dolgopyat D., Dong C., Kanigowski A., N\'andori P.
{\em Flexibility of statistical properties for smooth systems satisfying the central limit theorem,}
 arXiv:2006.02191.


\bibitem{DFSurv} 
Dolgopyat D., Fayad B. {\em Limit theorems for toral translations,} 
Proc. Sympos. Pure Math. {\bf 89} (2015) 227--277.

\bibitem{DLN19}
Dolgopyat D., Lenci M., N\'andori P.
{\em Global observables for random walks: law of large numbers,}  arXiv 1908.11504.

\bibitem{DN-Ren} Dolgopyat D., N\'andori P.
{\em Infinite measure renewal theorem and related results,}  
Bulletin LMS {\bf 51} (2019) 145--167.

\bibitem{DN-Flows} Dolgopyat D., N\'andori P.
{\em On mixing and the local central limit theorem for hyperbolic flows,} 
Erg. Th., Dynam. Sys. {\bf 40} (2020) 142--174.


\bibitem{DNP} Dolgopyat D., N\'andori P., P\`ene F.
{\em Asymptotic expansion of correlation functions for $\integers^d$ covers of hyperbolic flows,}
arXiv:1908.11504.

\bibitem{EL04} Einsiedler M., Lind D. {\em Algebraic $\integers^d$--actions of entropy rank 1}, Trans. AMS {\bf 356} (2004) 1799--1831. 


\bibitem{FFK} B.\ Fayad, G.\ Forni, A.\ Kanigowski, {\em Lebesgue spectrum of countable multiplicity for conservative flows on the torus},  
to appear in JAMS.


\bibitem{FMT07} Field M., Melbourne I., T\"or\"ok A. {\em Stability of mixing and rapid mixing for hyperbolic flows,}
  Ann. Math. {\bf 166} (2007) 269--291. 

\bibitem{FF03} Flaminio L., Forni G. {\em Invariant distributions and time averages for horocycle flows,} 
Duke Math. J. {\bf 119} (2003) 465--526.

\bibitem{FlaFo}Flaminio L., Forni G.,{\em Equidistribution of nilflows and applications to theta sums}. Erg. Th. Dyn. Sys {\bf 26} (2006) 409--433.

\bibitem{FP20} Fernando K., Pene F. 
{\em Expansions in the Central and the Local Limit Theorems for Dynamical Systems,}
arXiv:2008.08726

\bibitem{For02} Forni G. {\em Deviation of ergodic averages for area-preserving flows on surfaces of higher genus,} 
Ann. of Math. {\bf 155} (2002)  1--103. 

\bibitem{For16} Forni G. {\em Effective equidistribution of nilflows and bounds on Weyl sums,} 
London Math. Soc. Lecture Notes {\bf 437} (2016) 136--188.

\bibitem{FK17} Forni G., Kanigowski A.
{\em Time changes of Heisenberg nilflows,} 
 Asterisque {\bf 416} (2020) 253--299.
 
\bibitem{GS14} Gorodnik A., Spatzier R.
  {\em Exponential mixing of nilmanifold automorphisms,}
  J. Anal. Math. {\bf 123} (2014) 355--396.

\bibitem{GS15} Gorodnik A., Spatzier R.
{\em  Mixing properties of commuting nilmanifold automorphisms,} Acta Math. {\bf 215} (2015) 127--159.  

\bibitem{Gou04} Gou\"ezel S. {\em Sharp polynomial estimates for the decay of correlations,} 
Israel J. Math. {\bf 139} (2004) 29--65.

\bibitem{GM14} Gou\"ezel S., Melbourne I.
{\em Moment bounds and concentration inequalities for slowly mixing dynamical systems,} 
Electron. J. Probab. {\bf 19} (2014) paper 93.
  
\bibitem{Kal82} Kalikow S. A. {\em  $T,T^{-1}$ transformation is not loosely Bernoulli,} 
Ann. of Math. {\bf 115} (1982) 393--409.


\bibitem{KRHV} Kanigowski A.,  Rodriguez Hertz F., Vinhage K.
{\em On the non-equivalence of the Bernoulli and K properties in dimension four,} 
J. Mod. Dyn. {\bf 13} (2018) 221--250.

\bibitem{Kat80} Katok A. {\em Smooth non-Bernoulli K-automorphisms,} 
Invent. Math. {\bf 61} (1980) 291--299. 


\bibitem{Kat71} Katznelson Y. {\em Ergodic automorphisms of $\Tor^n$ are Bernoulli shifts,} Israel J. Math. {\bf 10} (1971) 186--195.

\bibitem{KS79} Kesten H., Spitzer F. 
{\em A limit theorem related to a new class of self-similar processes,} 
Z. Wahrsch. Verw. Gebiete {\bf 50} (1979) 5--25.

\bibitem{KM99} Kleinbock D. Y., Margulis G. A. {\em Logarithm laws for flows on homogeneous spaces,} Inv. Math. {\bf 138} (1999) 451--494.

\bibitem{LB06} Le Borgne S. 
{\em Exemples de systèmes dynamiques quasi-hyperboliques a decorrelations lentes,} 
C. R. Math. Acad. Sci. Paris {\bf 343} (2006), no. 2, 125--128. 

\bibitem{LBP05} Le Borgne S., P\`ene F. {\em Vitesse dans le theoreme limite central pour certains systemes dynamiques quasi-hyperboliques,} Bull. Soc. Math. France {\bf 133} (2005) 395--417.


  
\bibitem{Lind} Lind D., {\em Locally compact measure preserving flows}, 
Adv.  Math. {\bf 15} (1975)  175--190.

\bibitem{Liv04} Liverani C. {\em On contact Anosov flows,} Ann. of Math. {\bf 159} (2004) 1275--1312. 

\bibitem{MN79} Marcus B., Newhouse S.
  {\em Measures of maximal entropy for a class of skew products,}
  Lecture Notes in Math. {\bf 729} (1979) 105--125.
  
\bibitem{May} Mayer A. {\em Trajectories on the closed orientable surfaces}
Mat.
Sbornik {\bf 12} (1943) 71--84.


\bibitem{Mel09} Melbourne I.
{\em Large and moderate deviations for slowly mixing dynamical systems,} 
Proc. AMS {\bf 137} (2009) 1735--1741.

\bibitem{Mel18} Melbourne I. 
{\em Superpolynomial and polynomial mixing for semiflows and flows,} 
Nonlinearity {\bf 31} (2018) R268--R316.

\bibitem{MN08} Melbourne I., Nicol M.
{\em Large deviations for nonuniformly hyperbolic systems,} 
Trans. AMS {\bf 360} (2008) 6661--6676.

  
\bibitem{PP90} Parry W., Pollicott M. {\em Zeta Functions and Periodic Orbit Structure of Hyperbolic Dynamics,} Asterisque {\bf 187-188} (1990). 

\bibitem{Pene09} P\`ene F. {\em Planar Lorentz process in a random scenery,} 
Ann. Inst. Henri Poincare Probab. Stat. {\bf 45} (2009) 818--839.

\bibitem{RBY08} Rey--Bellet L., Young L.--S.
{\em Large deviations in non-uniformly hyperbolic dynamical systems,} 
Erg. Th. Dyn. Sys. {\bf 28} (2008) 587--612.

\bibitem{Rud88} Rudolph D.
{\em Asymptotically Brownian skew products give non-loosely Bernoulli K-automorphisms,} 
Invent. Math. {\bf 91} (1988) 105--128. 

\bibitem{Sar02}  Sarig O. 
{\em Subexponential decay of correlations,} 
Invent. Math. {\bf 150} (2002) 629--653. 

\bibitem{SV04} Sz\'asz D., Varj\'u T. {\em Local limit theorem for the Lorentz process and its recurrence in the plane,} 
Erg.Th. Dyn. Sys. {\bf 24} (2004) 257--278.

\bibitem{Ts18} Tsujii M. {\em Exponential mixing for generic volume-preserving Anosov flows in dimension three,} J. Math. Soc. Japan {\bf 70} (2018) 757--821.
  
  
  
\bibitem{You98} Young L.–S. {\em Statistical properties of dynamical systems with some hyperbolicity,} Ann. Math. {\bf 147} (1998) 585--650. 

\bibitem{You99} Young L.–S. {\em Recurrence times and rates of mixing,} Israel J. Math. {\bf 110} (1999) 153--188. 


\end{thebibliography}
\end{document}